\documentclass{amsart}
\usepackage{amssymb}
\usepackage{enumerate}
\usepackage{graphicx}
\usepackage{stmaryrd}
\usepackage{color}
\usepackage[dvipsnames]{xcolor}

\usepackage[colorlinks=true,
            linkcolor=violet,
            urlcolor=Aquamarine,
            citecolor=YellowOrange]{hyperref}

\newtheorem{thm}{Theorem}[section]

\newtheorem{lem}[thm]{Lemma}
\newtheorem{prop}[thm]{Proposition}
\newtheorem{defn}[thm]{Definition}
\newtheorem{rem}[thm]{Remark}

\newcommand{\norm}[1]{\left\Vert#1\right\Vert}
\newcommand{\abs}[1]{\left\vert#1\right\vert}
\newcommand{\set}[1]{\left\{#1\right\}}
\newcommand{\Real}{\mathbb R}

\newcommand{\pfrac}[2]{\frac{\partial #1}{\partial #2}}

\definecolor{darkgreen}{rgb}{0,0.5,0}
\definecolor{darkred}{rgb}{0.7,0,0}

\begin{document}

\title{Schauder estimates on smooth and singular spaces}

\author{Yaoting Gui and Hao Yin}

\begin{abstract}
In this paper, we present a proof of Schauder estimate on Euclidean space and use it to generalize Donaldson's Schauder estimate on space with conical singularities in the following two directions. The first is that we allow the total cone angle to be larger than 2$\pi$ and the second is that we discuss higher order estimates.
\end{abstract}
\address{Hao Yin,  School of Mathematical Sciences,
University of Science and Technology of China, Hefei, China}
\email{haoyin@ustc.edu.cn }
\address{Yaoting Gui,  School of Mathematical Sciences,
University of Science and Technology of China, Hefei, China}
\email{vigney@mail.ustc.edu.cn }
\maketitle
\tableofcontents

\section{Introduction}
In this paper, we discuss the classical Schauder estimate on Euclidean space $\Real^n$ and on some singular space with conical type singularities. The discussion contained in this paper should apply with minor modification to a class of conical type singular spaces, however, for simplicity, we restrict ourselves to a special case, namely, $\Real^{2}\times \Real^{n-2}$ with a singular metric
\begin{equation}\label{eqn:metric1}
	g_\beta= \abs{z}^{2\beta-2} dz^2 + d\xi^2, \quad \beta>0.
\end{equation}
where $z$ is in $\mathbb C$ (identified with $\Real^2$) and $\xi$ is in $\Real^{n-2}$. The geometry is nothing but the product of a 2-dimensional cone with cone angle $2\pi \beta$ and $\Real^{n-2}$. Our attention is drawn to this space because of the recent study of conical K\"ahler geometry proposed by \cite{tian1996kahler} and \cite{donaldson2012}. In the rest of this paper, we denote this space together with the metric by $X_\beta$.

We shall restrict ourselves to the interior estimates only and hopefully, the boundary value problem will be discussed in the future. Hence by Schauder estimate in $\Real^n$, we mean the inequality
\begin{equation*}
	\norm{u}_{C^{2,\alpha}(B_{1/2})}\leq C(n,\alpha) (\norm{f}_{C^\alpha(B_1)}+ \norm{u}_{C^0(B_1)})
\end{equation*}
if $\triangle u =f$ on $B_1$. If $f$ is in $C^{k,\alpha}$ for $k\in \mathbb N$, we can bound $\norm{u}_{C^{k+2,\alpha}}$ by successively taking derivatives and applying the above $C^{2,\alpha}$ estimate.

Besides the classical proof of potential theory, there are many different proofs by Campanato \cite{campanato}, Peetre \cite{peetre}, Trudinger \cite{trudinger}, Simon \cite{simon}, Safonov \cite{safonov1, safonov2}, Caffarelli \cite{caffarelli,CC} and Wang \cite{wang2006}. We refer the readers to \cite{wang2006} for brief comments on these proofs. Many of the above proofs have important applications to the study of nonlinear (or even fully nonlinear) elliptic and parabolic equations. The proof given below is motivated by the study of regularity problem on spaces with conical singularity. The ideas used here are related to the above mentioned proofs, for example, we shall use a characterization of H\"older continuous function known to Campanato and we shall compare the solution to polynomials as Caffarelli did in \cite{caffarelli,CC}. Moreover, the idea of pointwise Schauder estimate, due to Han \cite{han1998,han2000}, is particularly useful and effective for conical singularities. In the first part of this paper, we give a proof of the Schauder estimate on $\Real^n$. The proof is by far not as simple as the above mentioned ones. We need it, first to illustrate the basic idea of this paper, and second to prove some theorem that will be needed for the proof of Schauder estimate on $X_\beta$.

We start with an equivalent formulation of the H\"older space and the H\"older norm on $\Real^n$. Thanks to the Taylor expansion theorem, if $u$ is $C^{k,\alpha}$ in a neighborhood of $x$, then there exists a polynomial $P_x$ of degree $k$ such that
\begin{equation*}
	u(x+h)= P_x(h)+ O_x(\abs{h}^{k+\alpha}),\qquad \text{for} \quad \abs{h}<\delta_x.
\end{equation*}
It is natural to ask about the reverse: if a function $u$ has the above expansion around each point $x$ in an open set $\Omega$, is it true that $u\in C^{k,\alpha}(\Omega)$? As shown by the function
\begin{equation*}
	u(x)= x^2 \sin \frac{1}{x}, \qquad x\in \Real,
\end{equation*}
which is not $C^1$, we can not expect a positive answer without putting more restrictions to the expansion. It turns out that we need to ask the expansion to be uniform in the following sense: for some positive constants $\Lambda$ and $\delta$ independent of $x\in \Omega$,
\begin{itemize}
	\item the coefficients of the polynomial $P_x$  are bounded by $\Lambda$;
	\item the constant in the definition of $O_x(\abs{h}^{k+\alpha})$ is bounded by $\Lambda$;
	\item $\delta_x>\delta.$
\end{itemize}
The function $u$ that satisfies the above assumption is then said to have {\bf uniformly bounded expansion}, or UBE for simplicity.
It will be shown in Section \ref{sec:space} that the set of UBE functions is the same as $C^{k,\alpha}$ functions if one does not mind shrinking the domain a little, which is not a problem since we are only concerned with interior estimate in this paper. This allows us to translate the classical Schauder estimate on $\Real^n$ into a theorem about UBE functions.

A feature of the UBE characterization is that it seems to be a pointwise property. The proof of the Schauder estimate then reduces to showing that if $\triangle u =f$ and $f$ has an expansion of order $k+\alpha$ at $0$ bounded by $\Lambda$ in the above sense, then $u$ has an expansion at $0$ up to order $k+\alpha+2$ bounded by a constant multiple of $\Lambda$ and its own $C^0$ norm. This is exactly what we do in Section \ref{sec:another}.


Similar to what Han did in \cite{han1998,han2000}, an important step (Lemma \ref{lem:key}) is the following: let $f$ be $O(\abs{x}^{k+\alpha})$, then there exists $u$ that is $O(\abs{x}^{k+2+\alpha})$ satisfying
\begin{equation*}
	\triangle u =f \qquad \text{on} \quad B
\end{equation*}
and
\begin{equation*}
	\sup_{B\setminus \set{0}} \frac{\abs{u}}{\abs{x}^{k+2+\alpha}} \leq C \sup_{B\setminus \set{0}} \frac{\abs{f}}{\abs{x}^{k+\alpha}}.
\end{equation*}
This was proved by using the potential in \cite{han1998,han2000} and it is our intention to avoid using the potential, because the analysis of Green's function on $X_\beta$ could be complicated. Hence, we provide a proof of Lemma \ref{lem:key} using only the fact that harmonic functions are polynomials. This argument generalizes well on $X_\beta$.

\vskip 1cm

We then move on to the discussion of the singular space $X_\beta$. If $n$ is even and we identify $\Real^2\times \Real^{n-2}$ with $\mathbb C \times \mathbb C^{\frac{n-2}{2}}$, $X_\beta$ together with $g_\beta$ is also a (noncomplete) K\"ahler manifold. For $\beta\in (0,1)$ and $\alpha\in (0,\min\set{1,\frac{1}{\beta}-1})$, one can define $C^\alpha$ function, using the Riemannian distance as usual. Donaldson \cite{donaldson2012} observed that if one defines $C^{2,\alpha}_\beta$ space by requiring the function $u$, its gradient (in Riemannian geometric sense), its complex Hessian (in the above mentioned K\"ahler structure) to be $C^\alpha$, there is still a Schauder estimate. This estimate plays an important role in the study of conical K\"ahler geometry. Its original proof due to Donaldson is by potential theory and recently, there is another proof (without potential theory) of the same estimate by Guo and Song \cite{guo2016schauder}. Moreover, there is also a parabolic version of Donaldson's estimate due to Chen and Wang \cite{chen2015bessel}.

It is the main purpose of this paper to generalize the above Schauder estimate due to Donaldson in two directions (with a different proof). First, we allow any $\beta>0$ instead of $\beta\in (0,1)$. Second, we study regularity beyond second order derivatives. There are indeed situations in the study of conical K\"ahler geometry in which higher order regularity is necessary. See \cite{li2016uniqueness,zheng2017geodesics} for example. The case $\beta>1$ is useful, because the induced metric of some algebraic variety from its ambient space {\bf can be} conic with integer $\beta>1$. Applications along this line will be pursued in a separate paper.

This generalization is achieved by defining a new function space $\mathcal U^q$ on $X_\beta$(in Section \ref{sec:Holder}). Here $q$ is some positive number, taking the place of $k+\alpha$ for the usual $C^{k,\alpha}$ space. Briefly speaking, the definition is a combination of two ideas: first we require the function to be $C^{k,\alpha}$ for $q=k+\alpha$ away from the singularities; second, we use the idea in the first part of this paper, namely, we use {\bf uniformly bounded expansion} up to order $q$ to describe the regularity of $f$ at the singular points; finally, we need to take care of the transition between the two point of views. See (H1-H3) in Definition \ref{defn:big} for details.

We need to be clear about the type of expansion that is used near a singular point of $X_\beta$, because the Taylor expansion is not available here. This is the topic of Section \ref{subsec:formal}. On one hand, we need the expansion to be general so that it can be used to describe the regularity of the solution that we care; on the other hand, we want the expansion to be very special so that it contains as much information as possible. A choice of the expansion is a balance of the above two considerations. Our previous experience on the regularity issue of PDE's on conical spaces \cite{yin2016analysis,yin2016expansion} suggests that the good choice depends both on the parameter $\beta$ and on the type of equations that we are interested in. In order not to distract the attention of the readers, we give one particular choice in Section \ref{subsec:formal} by defining the $\mathcal T$-polynomial. This choice is sufficient to present the idea of our proof and it is general enough to have Donaldson's Schauder estimate as a special case. 
\begin{rem}
For future applications, we list a family of properties (P1-P4). As long as the definition of $\mathcal T$-polynomial satisfies (P1-P4), the Schauder estimate holds.
\end{rem}

Our choice is justified by the following theorems. They are the main results of this paper. The first is the Schauder estimate. Here $\mathcal D$ is a countable and discrete set of positive numbers (see Section \ref{sec:Holder}), $\hat{B}_r$ is the metric ball in $X_\beta$ with the origin as its center and $\mathcal U^q$ is the new space of functions given by Definition \ref{defn:big}.
\begin{thm}\label{thm:main1}
	Let $q>0$ and $q, q+2\notin \mathcal D$. Suppose $f\in \mathcal U^{q}(\hat{B}_2)$ and $u$ is a bounded weak solution to
	\begin{equation*}
		\triangle_\beta u =f.
	\end{equation*}
	Then $u$ is in ${\mathcal U}^{q+2}(\hat{B}_1)$ and
	\begin{equation*}
		\norm{u}_{{\mathcal U}^{q+2}(\hat{B}_1)}\leq C\left( \norm{u}_{C^0(\hat{B}_2)} + \norm{f}_{{\mathcal U}^{q}(\hat{B}_2)} \right).
	\end{equation*}
\end{thm}
The second is a comparison between the newly defined space $\mathcal U^q$ and the Donaldson's $C^{2,\alpha}_\beta$ space, whose definition we recall in Section \ref{subsec:compare}.
\begin{thm}\label{thm:main2} Suppose $0<\beta<1$ and $0<\alpha<\min \set{1,\frac{1}{\beta}-1}$. If we write $\mathcal X$ for $C^\alpha$ ($\mathcal U^\alpha$, $C^{2,\alpha}_\beta$ and $\mathcal U^{2+\alpha}$ respectively) and $\mathcal Y$ for $\mathcal U^\alpha$ ($C^\alpha$, $\mathcal U^{2+\alpha}$ and $C^{2,\alpha}_\beta$ respectively), then $u\in \mathcal X(\hat{B}_2)$  implies that $u\in \mathcal Y(\hat{B}_1)$ and
	\begin{equation*}
		\norm{u}_{\mathcal Y(\hat{B}_1)} \leq C \norm{u}_{\mathcal X(\hat{B}_2)}.
	\end{equation*}
\end{thm}

With Theorem \ref{thm:main2}, Theorem \ref{thm:main1} implies the Schauder estimate of Donaldson in \cite{donaldson2012}.

The rest of the paper is organized as follows. In Section \ref{sec:space}, we give a characterization of the $C^{k,\alpha}$ space on $\Real^n$ using uniformly bounded expansion. This was known to Campanato back to the 1960's. We include a proof for completeness, which may be omitted for a first reading. In Section \ref{sec:another}, we prove the Schauder estimate on $\Real^n$. These two sections form the first part of the paper. We then move on to the study on $X_\beta$. We first set up some notations and recall some easy facts about Poisson equations on $X_\beta$ in Section \ref{sec:pre}. In Section \ref{sec:harmonic}, we study bounded harmonic functions on $X_\beta$, which is key to the proof of Theorem \ref{thm:main1}. In Section \ref{sec:Holder}, we define the space $\mathcal U^q$ and prove Theorem \ref{thm:main2}. In the final section, we prove Theorem \ref{thm:main1}.

{\bf Acknowledgement.} The second author would like to thank Professor Xinan Ma for bringing the references \cite{han1998,han2000} to his attention.

\section{H\"older space on $\Real^n$}\label{sec:space}

In this section, we define a new space of functions that satisfy the uniform Taylor expansion condition and prove that it is equivalent to the usual H\"older space $C^{k,\alpha}$. As remarked in the introduction, this result is not new and the proofs are included for completeness.

We write $B_r$ for the ball of radius $r$ centered at the origin in $\Real^n$.
\begin{defn}\label{defn:ube}
Suppose $r$ and $\delta$ are two positive real numbers.	For a function $u$ defined on $B_{r+\delta}$, we say that it has {\bf uniformly bounded expansion} (or UBE for simplicity) up to order $q$ on $B_r$ with scale $\delta$ if there exists some $\Lambda>0$ such that for any $x\in B_r$ and $h\in B_\delta$
\begin{equation*}
	u(x+h)= p_x(h)+ O_x(q),
\end{equation*}
where $p_x(h)$ is a polynomial of $h$ whose coefficients (depending on $x$) are uniformly bounded by $\Lambda$ and $O_x(q)$ is also a function of $h$ satisfying
\begin{equation}\label{eqn:normO}
	\abs{O_x(q)(h)}\leq \Lambda \abs{h}^q,\quad \forall \abs{h}<\delta, \quad \forall x\in B_r.
\end{equation}
\end{defn}

Related to the above definition, we define the following notations:
\begin{enumerate}[(a)]
	\item
	The infimum of all $\Lambda$ satisfying \eqref{eqn:normO} is denoted by $[O_x(q)]_{O_q,B_\delta}$, which is nothing but
\begin{equation*}
	\sup_{h\in B_\delta, h\ne 0} \frac{\abs{O_x(q)}}{\abs{h}^q}.
\end{equation*}
	\item The set of all functions that have UBE up to order $q$ on $B_r$ with scale $\delta$ is denoted by $\mathcal U^{q,\delta}(B_r)$.
	\item For $u\in \mathcal U^{q,\delta}(B_r)$, the infimum of $\Lambda$ in the above definition is defined to be the norm of $u$, denoted by $\norm{u}_{\mathcal U^{q,\delta}(B_r)}$.
\end{enumerate}

It turns out that $\mathcal U^{q,\delta}$ is equivalent to the usual H\"older space $C^{k,\alpha}$ with $q=k+\alpha$ in the following sense.

\begin{prop}
	\label{prop:rn}Suppose $q=k+\alpha$ for some $k\in \mathbb N\cup \set{0}$ and $\alpha\in (0,1)$.

	(i) If $u\in C^{k,\alpha}(B_{r+\delta})$, then $u\in \mathcal U^{q,\delta}(B_r)$ and
	\begin{equation*}
		\norm{u}_{\mathcal U^{q,\delta}(B_r)}\leq C(\delta,q,r,n) \norm{u}_{C^{k,\alpha}(B_{r+\delta})}.
	\end{equation*}

	(ii) If $u\in \mathcal U^{q,\delta}(B_{r+\eta})$ for some $\eta>0$, then $u\in C^{k,\alpha}(B_{r})$ and
	\begin{equation*}
		\norm{u}_{C^{k,\alpha}(B_{r})}\leq C(\eta,q,\delta,r,n) \norm{u}_{\mathcal U^{q,\delta}(B_{r+\eta})}.
	\end{equation*}
\end{prop}

The rest of this section is devoted to the proof of this proposition. The first part follows trivially from the Taylor theorem with integral remainder.

The proof of (ii) is by induction and we assume without loss of generality that $r=1$. The starting point of the induction is the observation that the claim holds trivially true when $k=0$. For $k>0$, the expansion in Definition \ref{defn:ube} implies that $u$ is differentiable for each $x\in B_{1+\eta}$. Therefore, the proof of Proposition \ref{prop:rn} reduces to

{\bf Claim 1:} If $u\in \mathcal U^{q,\delta}(B_{1+\eta})$, then for $i=1,\cdots,n$, some $\eta'\in (0,\eta)$ and some $\delta'>0$,
\begin{equation*}
	\pfrac{u}{x_i} \in \mathcal U^{q-1,\delta'}(B_{1+\eta'}) \quad \text{and} \quad \norm{\pfrac{u}{x_i}}_{\mathcal U^{q-1,\delta'}(B_{1+\eta'})} \leq C(\delta,q,\eta,n) \norm{u}_{\mathcal U^{q,\delta}(B_{1+\eta})}.
\end{equation*}

For the proof of the claim, we recall some notations. Let $\epsilon=(\epsilon_1,\cdots,\epsilon_n)$ be a multi-index. For $h\in \Real^n$, we write
\begin{equation*}
	h^\epsilon= h_1^{\epsilon_1}\cdots h_n^{\epsilon_n}.
\end{equation*}
For some $\delta'>0$ to be determined in a minute, we {\bf fix} $h=(h_1,\cdots,h_n)$ satisfying $\abs{h}<\delta'$ and $h_i\ne 0$ for all $i$. Given this $h$ and a multi-index $\epsilon$ with $\abs{\epsilon}<q+2$, we define the difference quotient operator $P_{\epsilon,h}$ which maps a function defined on $B_{1+\eta}$ to a function defined on $B_{1+\eta'}$ as follows. For $\epsilon=(0,\cdots,1,\cdots,0)$, where the only $1$ is at the $i$-th position,
\begin{equation*}
	P_{\epsilon,h}[f](y):=\frac{f(y+h_ie_i)-f(y)}{h_i}
\end{equation*}
where $e_i$ is the natural basis of $\Real^n$. For $\epsilon=\epsilon'+e_i$, we define
\begin{equation*}
	P_{\epsilon,h}[f](y)=\frac{P_{\epsilon',h}[f](y+h_ie_i)-P_{\epsilon',h}[f](y)}{h_i}.
\end{equation*}
Since $\abs{\epsilon}$ is bounded by $q+2$, by choosing $\delta'$ small (say, $\delta'= \frac{\eta-\eta'}{2(q+2)}$, so that $\abs{h}$ small) depending on $\eta$ and $\eta'$, $P_{\epsilon,h}[f]$ is a function defined on $B_{1+\eta'}$.

\begin{lem}\label{lem:P}
	$P_{\epsilon,h}$ is well defined, i.e., it is independent of the order of induction in its definition. Moreover, we have
	\begin{equation}
		P_{\epsilon,h}[f](y)= \frac{1}{h^\epsilon}\sum_{0\leq \gamma\leq \epsilon} (-1)^{\abs{\gamma}+1} C^\gamma_\epsilon f(y+\gamma h).	
		\label{eqn:P}
	\end{equation}
	Here
	
	(a) $\gamma$ is a multi-index and $0\leq \gamma\leq \epsilon$ means that for each $i=1,\cdots,n$, we have $0\leq \gamma_i\leq \epsilon_i$;
	
	(b) $\gamma h= (\gamma_1h_1,\cdots,\gamma_n h_n)$;
	
	(c) $C^\gamma_\epsilon:= C^{\gamma_1}_{\epsilon_1} \cdots C^{\gamma_n}_{\epsilon_n}$.
\end{lem}

The proof is very elementary and omitted. We shall also need the following lemma about combinatorics.

\begin{lem}
	\label{lem:Pkill}
	For any multi-index $\sigma$, set
	\begin{equation*}
		Q^\sigma_\epsilon= \sum_{0\leq \gamma\leq \epsilon} (-1)^{\abs{\gamma}+1} C^\gamma_\epsilon \gamma^\sigma.
	\end{equation*}
	If $\sigma_i< \epsilon_i$ for some $i=1,\cdots,n$, then $Q^\sigma_\epsilon=0$. As a consequence, if we denote the multi-index $(\gamma_1-1,\cdots,\gamma_n-1)$ by $\gamma-1$, then for the same $\sigma$ and $\epsilon$ above
	\begin{equation*}
		Q^\sigma_\epsilon= \sum_{0\leq \gamma\leq \epsilon} (-1)^{\abs{\gamma}+1} C^\gamma_\epsilon (\gamma-1)^\sigma.
	\end{equation*}
\end{lem}
\begin{proof}We only prove the first claim of the lemma. An easy observation is that
	\begin{equation*}
		Q^\sigma_\epsilon=- \prod^n_{i=1} \sum_{0\leq \gamma_i\leq \epsilon_i} (-1)^{\gamma_i} C^{\gamma_i}_{\epsilon_i}\gamma_i^{\sigma_i}.
	\end{equation*}
	To show the product is zero, it suffices to show that the $i$-th factor is zero if $\sigma_i<\epsilon_i$.
	 Consider the polynomial
	\begin{equation*}
		f(y_i)=(1-y_i)^{\epsilon_i}=\sum_{0\leq \gamma_i\leq \epsilon_i} C^{\gamma_i}_{\epsilon_i} (-y_i)^{\gamma_i}.
	\end{equation*}
	For each $j=0,\cdots,\sigma_i$, since $j<\epsilon_i$, we have
	\begin{equation*}
		(\partial_{y_i})^{j} f|_{y_i=1}=0,
	\end{equation*}
	which gives
	\begin{equation*}
		\sum_{j\leq \gamma_i\leq \epsilon_i} (-1)^{\gamma_i} C^{\gamma_i}_{\epsilon_i} \frac{\gamma_i !}{(\gamma_i-j)!} =0.
	\end{equation*}
	By setting
	\begin{equation*}
		F(\gamma;j)=\gamma \cdot (\gamma-1) \cdot \cdots \cdot(\gamma-j+1),
	\end{equation*}
	we have
	\begin{equation}\label{eqn:qsb}
		\sum_{0\leq \gamma_i\leq \epsilon_i} (-1)^{\gamma_i} C^{\gamma_i}_{\epsilon_i} F(\gamma_i;j) =0.
	\end{equation}
	$F(\gamma;j)$ is a polynomial of $\gamma$ of degree $j$. Since $\sigma_i<\epsilon_i$,  $\gamma_i^{\sigma_i}$ is then a linear combination of $F(\gamma_i;j)$, $j=0,\cdots,\epsilon_i$. With \eqref{eqn:qsb}, this concludes the proof of Lemma \ref{lem:Pkill}.

\end{proof}

\begin{rem}
	The above lemma is related to the fact that difference quotient kills polynomials.
\end{rem}

Now, we come back to the proof of Claim 1, which consists of two steps. In the first step, we restrict ourselves to a special type of $h$ satisfying
\begin{equation}\label{eqn:betterh}
	\abs{h_i}\geq \frac{1}{2\sqrt{n}} \abs{h},\quad \text{for}\, i=1,\cdots,n.
\end{equation}
 We denote the set of such $h$ by $\Omega$. The reason will be clear in a minute.

\begin{defn}
	\label{defn:pube} Suppose $r$ and $\delta$ are two positive real numbers. A function $u:B_{r+\delta}\to \Real$ is said to have {\bf partially uniformly bounded expansion} with respect to $\Omega$, up to order $q$ and with scale $\delta$, if the assumptions in Definition \ref{defn:ube} hold with $h\in B_\delta$ replaced by $h\in \Omega\cap B_\delta$.

	The space of these functions is denoted by $\mathcal U^{p,\delta}_\Omega(B_{r})$ and its norm by $\norm{\cdot}_{\mathcal U^{p,\delta}_\Omega}$.
\end{defn}

The goal of the first step is the following claim, which is a partial version of Claim 1.

{\bf Claim 2:} If $u\in \mathcal U_\Omega^{q,\delta}(B_{1+\eta})$, then for $i=1,\cdots,n$, some $\eta'\in (0,\eta)$ and some $\delta'>0$,
\begin{equation*}
	\pfrac{u}{x_i} \in \mathcal U_\Omega^{q-1,\delta'}(B_{1+\eta'}) \quad \text{and} \quad \norm{\pfrac{u}{x_i}}_{\mathcal U_\Omega^{q-1,\delta'}(B_{1+\eta'})} \leq C(\delta,q,\eta,n,\Omega) \norm{u}_{\mathcal U_\Omega^{q,\delta}(B_{1+\eta})}.
\end{equation*}

In the second step, we shall derive Claim 1 from Claim 2. For the proof of Claim 2, recall that
\begin{equation}\label{eqn:Pbeta}
	P_{\epsilon,h}[f](x)=\frac{1}{h^\epsilon} \sum_{0\leq \gamma\leq \epsilon} (-1)^{\abs{\gamma}+1} C^\gamma_\epsilon f(x+\gamma h).
\end{equation}
Using the partial UBE assumption, we may expand $f(x+\gamma h)$ into a polynomial centered at $x$,
\begin{equation}
	f(x+\gamma h)=f(x)+\sum_{\abs{\sigma}<q} a_\sigma(x) (\gamma h)^\sigma + O(\abs{h}^q).	
	\label{eqn:fx}
\end{equation}
We may also use the expansion centered at $x+h$ to get
\begin{equation}
	f(x+\gamma h) =f(x+h)+\sum_{\abs{\sigma}<q} a_\sigma(x+h) ( (\gamma-1) h)^\sigma + O(\abs{h}^q).	
	\label{eqn:fxh}
\end{equation}
If we plug both \eqref{eqn:fx} and \eqref{eqn:fxh} into \eqref{eqn:Pbeta} and notice that $(\gamma h)^\sigma= \gamma^\sigma h^\sigma$, Lemma \ref{lem:Pkill} implies that
\begin{equation}\label{eqn:gooda}
	\sum_{\epsilon\leq \sigma, \abs{\sigma}<q} a_\sigma(x) Q^\sigma_\epsilon h^{\sigma-\epsilon} = \sum_{\epsilon\leq \sigma, \abs{\sigma}<q} a_\sigma(x+h) Q^\sigma_\epsilon h^{\sigma-\epsilon} + O(\abs{h}^{q-\abs{\epsilon}}).
\end{equation}
Here we also used the fact that
\begin{equation*}
	\frac{O(\abs{h}^q)}{h^\epsilon}
\end{equation*}
is an $O(\abs{h}^{q-\abs{\epsilon}})$, which is true because $h\in \Omega$. In fact, this is the only place we use the restriction $h\in \Omega$ in the proof of Claim 2.

We learn from \eqref{eqn:gooda} that for any multi-index $\epsilon$ with $1\leq \abs{\epsilon}<q$,
\begin{equation}\label{eqn:bettera}
	a_\epsilon(x+h) = a_\epsilon(x) + \bar{P}_{x,\epsilon}(h) + O(\abs{h}^{q-\abs{\epsilon}}),
\end{equation}
where $\bar{P}_{x,\epsilon}$ is a polynomial of $h$, whose coefficients (depending on $x$ and $\epsilon$) is uniformly bounded.
When $\abs{\epsilon}=[q]$, \eqref{eqn:bettera} is the same as \eqref{eqn:gooda} (with $\bar{P}_{x,\epsilon}=0$). When $\abs{\epsilon}<[q]$, we prove \eqref{eqn:bettera} by induction and assume that \eqref{eqn:bettera} is known for all $a_\sigma$ if $\abs{\sigma}>\abs{\epsilon}$. We rewrite \eqref{eqn:gooda}
\begin{equation*}
	Q^\epsilon_\epsilon(a_\epsilon(x+h)-a_\epsilon(x))= \sum_{\epsilon\lneq \sigma, \abs{\sigma}<q} \left( a_\sigma(x)-a_\sigma(x+h) \right) Q^\sigma_\epsilon h^{\sigma-\epsilon} + O(\abs{h}^{q-\abs{\epsilon}}).
\end{equation*}
By induction hypothesis, we insert \eqref{eqn:bettera} for $\sigma\gneq \epsilon$ into the above equation to get \eqref{eqn:bettera} for $\epsilon$.

If $\abs{\epsilon}=1$, then $a_\epsilon(x)$ is nothing but the partial derivative of $u$ at $x$ and \eqref{eqn:bettera} is the desired estimate in Claim 2.

Next, we show how to obtain Claim 1 from Claim 2, exploiting the rotational symmetry of the statement. For any orthogonal $n\times n$ matrix $A$, set
\begin{equation*}
	\tilde{f}(y)=f(Ay), \quad \text{or equivalently,} \quad f(x)=\tilde{f}(A^{-1} x).
\end{equation*}
Assume $f$ is in $\mathcal U^{q,\delta}(B_{1+\eta})$ (as assumed in Claim 1). Then $\tilde{f}\in \mathcal U^{q,\delta}(B_{1+\eta})\subset \mathcal U^{q,\delta}_\Omega(B_{1+\eta})$. By Claim 2, which has been proved, for any $y\in B_{1+\eta'}$ and $h\in \Omega\cap B_{\delta'}$, we have, uniformly,
\begin{equation}\label{eqn:goodpartial}
	\pfrac{\tilde{f}}{y_i}(y+h) = \pfrac{\tilde{f}}{y_i}(y) + \sum_{0<\abs{\epsilon}<q} \tilde{a}_\epsilon(y) h^\epsilon + O(\abs{h}^{q-1}).
\end{equation}
By Claim 2, the chain rule and \eqref{eqn:goodpartial}, as long as $A^{-1}h\in \Omega$, we have
\begin{eqnarray*}
	\pfrac{f}{x_j}(x+h)&=&  (A^{-1})^i_j \pfrac{\tilde{f}}{y_i} \left( A^{-1} x + A^{-1}h \right) \\
	&=& (A^{-1})^i_j \left[ \pfrac{\tilde{f}}{y_i}(A^{-1} x) + \sum_{0<\abs{\epsilon}<q} \tilde{a}_\epsilon(A^{-1}x) (A^{-1} h)^\epsilon + O(\abs{h}^{q-1}) \right] \\
	&=& \pfrac{f}{x_j}(x) + \sum_{0<\abs{\epsilon}<q} \hat{a}_{\epsilon,A}(x) h^\epsilon + O(\abs{h}^{q-1}).
\end{eqnarray*}

In summary, we have proved that $\pfrac{f}{x_i}$ is partially UBE with respect to $A\Omega$ up to order $q-1$. Now, we take orthogonal matrices $A_1,\cdots,A_l$ such that
\begin{equation*}
	\Real^n = A_1\Omega \bigcup \cdots \bigcup A_l \Omega.
\end{equation*}
Then the partial UBE conditions for each $k$ combine to be UBE if we can justify that
\begin{equation*}
	\tilde{a}_{\epsilon,A_k}(x)
\end{equation*}
is independent of $k=1,\cdots,l$. This is true because we can choose $A_k$ so that $A_{k_1}\Omega \cap A_{k_2}\Omega$ is either empty or has non-empty interior.

\section{A proof of the Schauder estimates on $\Real^n$}\label{sec:another}

We give another proof to the well-known interior Schauder estimate in this section.

Given Proposition \ref{prop:rn}, it suffices to prove
\begin{thm}[Schauder estimate] \label{thm:rn}Suppose that $f\in \mathcal U^{q,\delta}(B_1)$ for some $q=k+\alpha$ with $k\in \mathbb N$ and $\alpha\in (0,1)$. If $u$ is a bounded solution to $\triangle u =f$ on $B_1$, then $u$ lies in $\mathcal U^{q+2,\delta}(B_{1-\delta})$ and
	\begin{equation}\label{eqn:circlestar}
		\norm{u}_{\mathcal U^{q+2,\delta}(B_{1-\delta})} \leq C(n,q,\delta) (\norm{f}_{\mathcal U^{q,\delta}(B_1)} + \norm{u}_{C^0(B_1)}).
	\end{equation}
\end{thm}

For the proof, we need the following lemma (see Lemma 2.1 in \cite{han2000}),
\begin{lem}
	\label{lem:key}
	If $f:B_{r}\to \Real$ is $O(q)$ and $q$ is not an integer, then there exists some $u\in O(2+q)$ satisfying
	\begin{equation*}
		\triangle u =f\quad \text{on} \quad B_r.
	\end{equation*}
	Moreover, for some $C>0$ depending on $n,q,r$,
	\begin{equation*}
		[u]_{O_{q+2},B_r}\leq C [f]_{O_q,B_r}.
	\end{equation*}
\end{lem}

Before the proof of Lemma \ref{lem:key}, we show how Theorem \ref{thm:rn} follows from it.  For any $x\in B_{1-\delta}$ fixed, there exists a polynomial $p_{f,x}(h)$ (of order $k$) such that
\begin{equation*}
	f(x+h) = p_{f,x}(h) + e_{f,x}(h) \quad \text{on} \quad B_\delta,
\end{equation*}
where $e_{f,x}$ is $O(q)$. By the definition, all the coefficients of $p_{f,x}$ are bounded by $\norm{f}_{\mathcal U^{q,\delta}(B_1)}$. Hence there exists another polynomial $p_{u,x}(h)$ (of order $k+2$, not unique) whose coefficients are bounded by $C(n,q) \norm{f}_{\mathcal U^{q,\delta}(B_1)}$ such that
\begin{equation*}
	\triangle p_{u,x}= p_{f,x} \quad \text{on} \quad B_\delta.
\end{equation*}

By Lemma \ref{lem:key}, there is some $e_{u,x} \in O(q+2)$ such that $\triangle e_{u,x}= e_{f,x}$ on $B_\delta$. Therefore,
\begin{equation}\label{eqn:upe}
	\triangle (u(x+h) - p_{u,x}(h) - e_{u,x}(h)) =0\quad \text{on} \quad B_\delta.
\end{equation}
Moreover, also by Lemma \ref{lem:key}, 
\begin{equation*}
	[e_{u,x}]_{O_{q+2},B_\delta} \leq C [e_{f,x}]_{O_q,B_\delta}.
\end{equation*}
In particular, $\norm{e_{u,x}}_{C^0(B_\delta)}$ is bounded by a multiple of $\norm{f}_{\mathcal U^{q,\delta}(B_1)}$.

By \eqref{eqn:upe}, $u(x+h)-p_{u,x}(h)-e_{u,x}(h)$ is a bounded harmonic function that is bounded on $B_\delta$ by the right hand side of \eqref{eqn:circlestar}. By well-known properties of harmonic functions,
\begin{equation*}
	u(x+h)-p_{u,x}(h)-e_{u,x}(h)=\tilde{p}(h) + \tilde{e}(h),
\end{equation*}
where $\tilde{p}(h)$ is a polynomial of order $k+2$ and $\tilde{e}(h)$ is $O_{q+2,B_\delta}$ and again, the coefficients of $\tilde{p}$ and $[\tilde{e}]_{O_{q+2,B_\delta}}$ is bounded by the right hand side of \eqref{eqn:circlestar}.

By setting $\tilde{p}_{u,x}=p_{u,x}+\tilde{p}$ and $\tilde{e}_{u,x}=e_{u,x}+\tilde{e}$, we have
\begin{equation*}
	u(x+h)=\tilde{p}_{u,x}(h)+ \tilde{e}_{u,x}(h).
\end{equation*}
Notice that the constants in the above argument are independent of $x\in B_{1-\delta}$, hence we have verified that $u\in \mathcal U^{q+2,\delta}(B_{1-\delta})$ with the desired bound.

\vskip 1cm
The rest of this section is devoted to the proof of Lemma \ref{lem:key}. Without loss of generality, we assume $\delta=1$.

We decompose $B_1$ into the union of a sequence of annulus
\begin{equation*}
	A_l:= B_{2^{-l}} \setminus B_{2^{-l-1}} \quad \text{for} \quad l=0,1,2,\cdots.
\end{equation*}
Set
\begin{equation*}
	f_l= f\cdot \chi_{A_l},
\end{equation*}
where $\chi_{A_l}$ is the characteristic function of $A_l$. For simplicity, in the rest of this proof, we write $\Lambda_f$ for $[f]_{O_q,B_\delta}$. By definition,
\begin{equation*}
	\abs{f_l(x)}\leq \Lambda_f \abs{x}^q \chi_{A_l} \leq \Lambda_f 2^{-lq} \quad \text{on} \quad \Real^n.
\end{equation*}
Let $w_l$ be the unique solution (vanishing at the infinity) to the Poisson equation $\triangle w_l =f_l$ on $\Real^n$. Obviously, $w_l$ is harmonic in the complement of $A_l$. Moreover, we have the uniform bound
\begin{equation}\label{eqn:wl}
	\sup_{\Real^n} \abs{w_l} \leq C \Lambda_f 2^{-l (q+2)}.
\end{equation}
Since $w_l$ is harmonic in $B_{2^{-l-1}}$, it is a converging power series there and let $P_l$ be the polynomial that is the part of this series with order strictly smaller than $q+2$, namely, if
\begin{equation*}
	w_l(x)= \sum_\epsilon a_\epsilon x^\epsilon
\end{equation*}
then
\begin{equation}\label{eqn:defnPl}
P_l(x) = \sum_{\abs{\epsilon}<q+2} a_\epsilon x^\epsilon.
\end{equation}
Using the fact that $w_l$ is harmonic on $B_{2^{-l-1}}$ and bounded by $C\Lambda_f 2^{-l(q+2)}$, we estimate
\begin{equation}\label{eqn:abeta}
	\abs{a_\epsilon}\leq C_\epsilon \Lambda_f 2^{-l(q+2-\abs{\epsilon})}.
\end{equation}
Here $C_\epsilon$ is a constant depending on $n$ and $\epsilon$.

Setting
\begin{equation*}
	u_l(x) = w_l(x) - P_l(x),
\end{equation*}
we have the following estimates,
\begin{lem}
	\label{lem:ul} There exists a constant $C_q$ depending on $n$ and $q$ such that

	(i) on $B_{2^{-l-1}}$,
	\begin{equation*}
		\abs{u_l(x)} \leq C_q\Lambda_f 2^{(1+[q]-q)l}\abs{x}^{[q]+3};
	\end{equation*}

	(ii) on $A_l$,
	\begin{equation*}
		\abs{u_l(x)} \leq C_q \Lambda_f 2^{-(q+2)l}
	\end{equation*}
	or equivalently
	\begin{equation*}
		\abs{u_l(x)} \leq C_q \Lambda_f \abs{x}^{q+2};
	\end{equation*}

	(iii) on $B_1\setminus B_{2^{-l}}$,
	\begin{equation*}
		\abs{u_l(x)} \leq C_q \Lambda_f \left[ 2^{-l(q+2)} + \sum_{\abs{\epsilon}<q+2} 2^{-l(q+2-\abs{\epsilon})} \abs{x}^{\abs{\epsilon}} \right].
	\end{equation*}
\end{lem}
Here and in the following proof, $C_q$ may vary from line to line, as long as it depends only on $n$ and $q$.
\begin{proof}
	First, we estimate $P_l(x)$ in $B_{2^{-l}}$ as follows
	\begin{eqnarray*}
		\abs{P_l(x)} &\leq& \sum_{ \abs{\epsilon}<q+2} \abs{a_\epsilon} \abs{x}^{\abs{\epsilon}} \\
		&\leq& C_q \Lambda_f \sum_{\abs{\epsilon}<q+2} 2^{-l(q+2-\abs{\epsilon})} 2^{-l\abs{\epsilon}} \\
		&\leq& C_q \Lambda_f 2^{-l(q+2)}.
	\end{eqnarray*}
	Here we have used \eqref{eqn:abeta} and $\abs{x}\leq 2^{-l}$.
	Together with \eqref{eqn:wl}, this implies that
	\begin{equation}
		\abs{u_l(x)} \leq C_q \Lambda_f 2^{-(q+2)l}\qquad \text{on}\quad B_{2^{-l}},
		\label{eqn:ul2}
	\end{equation}
	which in particular proves (ii) of Lemma \ref{lem:ul}.

	By the definition of $P_l$, $u_l$ is a harmonic function on $B_{2^{-l}}$ which vanishes at the origin up to order $[q]+2$. Hence, (i) of Lemma \ref{lem:ul} follows from \eqref{eqn:ul2} and a scaled version of the following fact:

	If $u$ is a harmonic function on $B_1$ bounded by $1$ and vanishes at the origin up to order $k$, then there is a universal constant depending only on the dimension such that
	\begin{equation*}
		\abs{u(x)}\leq C_n \abs{x}^{k+1}.
	\end{equation*}

	The last part of Lemma \ref{lem:ul} is a trivial combination of \eqref{eqn:wl} and \eqref{eqn:abeta}.
\end{proof}

With these preparations, we claim that
\begin{equation}\label{eqn:defu}
	u(x) =\sum_{l=0}^\infty u_l(x)
\end{equation}
converges on $B_1$ and gives the desired solution $u$ in Lemma \ref{lem:key}.

Since $w_l$ is harmonic in a neighborhood of $0\in \Real^n$, $P_l$ defined in \eqref{eqn:defnPl} is a harmonic polynomial on the entire $\Real^n$. As a consequence,
\begin{equation*}
	\triangle u_l = f_l \quad \text{on} \quad \Real^n.
\end{equation*}
Hence, to show Lemma \ref{lem:key}, it suffices to prove
\begin{equation}\label{eqn:convergeu}
	\sum_{l=0}^\infty \abs{u_l}(x) \leq C_q \Lambda_f \abs{x}^{q+2} \qquad \text{on} \quad B_1,
\end{equation}
which not only implies the convergence of \eqref{eqn:defu}, but also gives the expected bound of $u$ in Lemma \ref{lem:key}. For each $x\in B_1\setminus \set{0}$, all but finitely many $u_l$'s are harmonic in a neighborhood of $x$, hence, the convergence is smooth and $u$ satisfies the Poisson equation $\triangle u=f$.

For each fixed $x\in B_1\setminus \set{0}$, let $l_0$ be given by the condition that
\begin{equation*}
	x\in A_{l_0}.
\end{equation*}
In other words, $\abs{x} < 2^{-l_0} \leq 2\abs{x}$.

To estimate the left hand side of \eqref{eqn:convergeu}, we compute
\begin{eqnarray*}
	\sum_{l=0}^{l_0-1} \abs{u_l}(x) &\leq& C_q \Lambda_f \left( \sum_{l=0}^{l_0-1} 2^{(1+[q]-q)l} \right) \abs{x}^{[q]+3} \\
	&\leq& C_q \Lambda_f 2^{(1+[q]-q)l_0} \abs{x}^{[q]+3} \\
	&\leq& C_q \Lambda_f \abs{x}^{q+2},
\end{eqnarray*}
where we used (i) in Lemma \ref{lem:ul}. (ii) of Lemma \ref{lem:ul} implies
\begin{equation*}
	\abs{u_{l_0}(x)} \leq C_q \Lambda_f \abs{x}^{q+2}.
\end{equation*}
Similarly, using (iii) of Lemma \ref{lem:ul}, we have
\begin{eqnarray*}
	\sum_{l>l_0} \abs{u_l}(x) &\leq& C_q \Lambda_f \sum_{l>l_0} \left[ 2^{-l(q+2)} + \sum_{\abs{\epsilon}<q+2} 2^{-l(q+2-\abs{\epsilon})} \abs{x}^{\abs{\epsilon}} \right]\\
	&\leq& C_q \Lambda_f \abs{x}^{q+2}.
\end{eqnarray*}
This finishes the proof of \eqref{eqn:convergeu} and hence the proof of Lemma \ref{lem:key}.

\section{Preliminaries about $X_\beta$}\label{sec:pre}
In this section, we first define some notations and then recall some basic properties about the Poisson equation on $X_\beta$ whose proofs are omitted.

\subsection{Notations}\label{subsec:notation}
Aside from the natural coordinates $(x,y,\xi)$ of $X_\beta= \Real^2\times \Real^{n-2}$, there is a global coordinate system $(\rho,\theta,\xi)$ on the smooth part of $X_\beta$ given by
\begin{equation*}
	\rho=\frac{1}{\beta} r^{\beta}, \quad x=r\cos\theta,\quad y=r\sin \theta.
\end{equation*}
In terms of $(\rho,\theta,\xi)$, the metric $g_\beta$ in \eqref{eqn:metric1} becomes
\begin{equation*}
	g_\beta= d\rho^2 + \rho^2 \beta^2 d\theta^2 + d\xi^2.
\end{equation*}

Here is a list of notations that are useful.
\begin{enumerate}[(i)]
	\item The singular set, denoted by $\mathcal S$, corresponds to $\set{\rho=0}$ and can be parametrized by $\xi$.
	\item $d(x,y)$ is the Riemannian distance (given by $g_\beta$) between $x$ and $y$ in $X_\beta$.

	\item $\mathcal S_\delta$ is the set of points whose distance to $\mathcal S$ is smaller than $\delta>0$.
	\item $\Omega_\delta= X_\beta \setminus \mathcal S_\delta$.
	\item For a point $x\in X_\beta$, $\hat{B}_r(x)$ is the set of points whose distance to $x$ is smaller than $r$.
	\item
$x_0$ is the origin of $X_\beta$, i.e. the point with $\rho=0$ and $\xi=0$. $\hat{B}_1(x_0)$ is the unit ball, which for simplicity is often denoted by $\hat{B}_1$.
\item $d(x,x_0)$ is usually denoted by $d(x)$.
\item
$\tilde{x}_0$ is the point in $X_\beta$ with $\rho=1$, $\theta=0$ and $\xi=0$. Then there is a constant $c_\beta$ depending only on $\beta$ such that $\hat{B}_{c_\beta}(\tilde{x}_0)$ is topologically a ball and that the restriction of $g_\beta$ to it is comparable with the flat metric on $B_{c_\beta}(0)\subset \Real^n$. Throughout this paper, we fix this $c_\beta$ and denote $\hat{B}_{c_\beta}(\tilde{x}_0)$ by $\tilde{B}$, which also serves as a unit ball. We also write $\tilde{B}_{r}$ for $\hat{B}_{c_\beta r}(\tilde{x}_0)$

\end{enumerate}

For each $x\in X_\beta\setminus \mathcal S$, there is a natural scaling map $\Psi_x$ which maps $\hat{B}_{c_\beta \rho(x)}(x)$ to $\tilde{B}$, where $\rho(x)$ is the $\rho$ coordinate of $x$, i.e. the distance to $\mathcal S$. If $x=(\rho_0,\theta_0,\xi_0)$, then
\begin{equation*}
	\Psi_x(\rho,\theta,\xi)= (\frac{\rho}{\rho_0}, \theta-\theta_0, \frac{\xi-\xi_0}{\rho_0}).
\end{equation*}
Here $\theta-\theta_0$ is understood as the natural minus operation of the group $S^1$.  The scaling $\Psi_x$ induces a pushforward of functions, which we denote by $S_x$. More precisely, if $u$ is a function defined on $\hat{B}_{c_\beta \rho(x)}(x)$, then 
\begin{equation*}
	S_x(u)(y)= u(\Psi_x^{-1}(y))\qquad \forall y\in \tilde{B}.
\end{equation*}

\subsection{Basics on the Poisson equation}\label{subsec:basic}
We collect a few basic facts about PDEs on $X_\beta$.

We start with an observation that is known and utilised by many authors. Consider another copy of $\Real^n$, whose coordinates are given by $(w,v,\xi)$, where $w,v\in \Real$ and $\xi\in \Real^{n-2}$. The Euclidean metric on $\Real^n$ is given by
\begin{equation*}
	dw^2+dv^2+d\xi^2.
\end{equation*}
One can check by direct computation that the mapping
\begin{equation*}
	(\rho,\theta,\xi) \mapsto (\rho\cos\theta, \rho \sin\theta, \xi)\in \Real^{n}
\end{equation*}
is bi-Lipschitzian from $X_\beta$ to $\Real^n$. Hence, the Sobolev space $W^{1,2}(X_\beta)$ ($L^p(X_\beta)$) is the same set of functions as $W^{1,2}(\Real^n)$ ($L^p(\Real^n)$). Moreover, the Sobolev inequality on $X_\beta$ holds with a different constant.

One can prove the following by the usual variation method and Moser iteration. Please note that we state it in a scaling invariant form.
\begin{lem}
	\label{lem:poisson}
	Let $f$ be an $L^\infty$ function supported in $\hat{B}_r$. There exists a solution $u\in W^{1,2}_{loc}(X_\beta)\cap L^\infty(X_\beta)$ to the Poisson equation
	\begin{equation*}
		\triangle_\beta u =f
	\end{equation*}
	with the bound
	\begin{equation*}
		\norm{u}_{L^\infty(X_\beta)}\leq C r^2 \norm{f}_{L^\infty(\hat{B}_r)}.
	\end{equation*}
\end{lem}

\section{Bounded harmonic functions on $X_\beta$}\label{sec:harmonic}
Suppose that $u$ is a bounded harmonic function on $\hat{B}_2\subset X_\beta$, i.e.,
\begin{equation*}
	\triangle_\beta u =0.
\end{equation*}
We discuss in this section the regularity of $u$ in $\hat{B}_{1}$. Before our discussion on $X_\beta$, we recall that if $u$ is a harmonic function on $B_1\subset \Real^n$, then we can bound any derivatives of $u$ on $B_{1/2}$ by the $C^0$ norm of $u$ on $B_1$. Equivalently, there is the Taylor expansion of $u$ at $0$,
\begin{equation}\label{eqn:expansion1}
	u(x)= \sum_{\abs{\sigma}\leq k} a_\sigma x^\sigma + O(\abs{x}^{k+1}),
\end{equation}
where $\sigma$ is a multi-index and the $a_\sigma$'s and the constant in the definition of $O(\abs{x}^{k+1})$ are bounded by the $C^0$ norm of $u$ on $B_1$.
The first goal of this section is to prove a generalization of this result for harmonic functions on $X_\beta$. More precisely, we prove an analog of \eqref{eqn:expansion1} for harmonic function on $X_\beta$ and provide estimates for the coefficients in the expansion.

Although it is very likely that the expansion we prove below (see Proposition \ref{prop:stronghf}) gives a converging series if we trace the bound for coefficients in the expansion, we do not pursue it here. However, we need the fact that the approximation polynomial of a harmonic function (as given in Proposition \ref{prop:stronghf}) is still harmonic as in the case of $\Real^n$. This is the second goal of this section and is contained in the second subsection.

\subsection{The expansion}\label{sub:expansion}
The first thing we need to do is to generalize the concept of polynomial (or monomial) used in the Taylor expansion. For the time being, we concentrate on the expansion describing the regularity of harmonic functions. In terms of the polar coordinates $(\rho,\theta,\xi)$,  the (monic) {\bf $X_\beta$-monomials} are defined to be
\begin{equation*}
	\rho^{2j+\frac{k}{\beta}} \cos k\theta \xi^\sigma, \rho^{2j+\frac{k}{\beta}}\sin k \theta \xi^\sigma,
\end{equation*}
for each multi-index $\sigma$ of $\Real^{n-2}$ and any $j,k\in \mathbb N\cup \set{0}$. Then sum $2j+\frac{k}{\beta}+\abs{\sigma}$ is called the {\bf degree of the $X_\beta$-monomial} and an {\bf $X_\beta$-polynomial} is a finite linear combination of monomials.

\begin{rem}
	Throughout this paper, unless stated otherwise, the range of $j,k$ and $\sigma$ in a summation is understood as above.
\end{rem}

With these definitions, we can now state the main result of this subsection.

\begin{prop}\label{prop:stronghf}
	If $u$ is a bounded harmonic function in $\hat{B}_1$, then for any $q>0$ and $d(x)<1/2$,
	\begin{equation}\label{eqn:hfexpansion}
		u(x)= \sum_{2j+\frac{k}{\beta}+\abs{\sigma}<q} \rho^{2j+\frac{k}{\beta}} \left( a_{j,k}^\sigma \cos k\theta  +  b_{j,k}^\sigma  \sin k\theta \right) \xi^\sigma + O(d(x)^q),
	\end{equation}
	where by definition $\abs{O(d(x)^q)}\leq \Lambda d(x)^q$.
	Moreover,
	\begin{equation*}
		\abs{a_{j,k}^\sigma}, \abs{b_{j,k}^\sigma}, \Lambda \leq C(q) \norm{u}_{C^0(\hat{B}_1)}.
	\end{equation*}
\end{prop}

The rest of this subsection is devoted to the proof of this result. First, we notice that the regularity of $u$ in $\xi$-direction is not a problem. This is summarized in the next lemma, whose proof is omitted.
\begin{lem}\label{lem:xigood}
	Suppose that $u$ is a weak harmonic function on $\hat{B}_2$. Then for any multi-index $\sigma$, $\partial_\xi^\sigma u$ is also a weak harmonic function on $\hat{B}_2$ and
	\begin{equation*}
		\norm{\partial_\xi^\sigma u}_{C^0(\hat{B}_1)}\leq C(\abs{\sigma}) \norm{u}_{C^0(\hat{B}_2)}.
	\end{equation*}
\end{lem}

When $\xi$ is fixed, we write $u(\xi)$ for $u$ as a function of $\rho,\theta$. The regularity of $u(\xi)$ is essentially a two dimensional problem that has been studied in \cite{yin2016analysis}.
The proof in \cite{yin2016analysis} yields
\begin{lem}
	\label{lem:rhogood}
	For $\abs{\xi}<1/2$,
\begin{equation}\label{eqn:expansionu}
	u(\xi)= \sum_{2j+\frac{k}{\beta}<q} \rho^{2j+\frac{k}{\beta}} \left( a_{j,k}(\xi) \cos k\theta + b_{j,k}(\xi) \sin k\theta \right) + O(\rho^q)\qquad \text{for} \quad  \rho<1/2.
\end{equation}
	
	For any multi-index $\sigma$,
\begin{equation}\label{eqn:expansionusigma}
	\partial^\sigma_\xi u(\xi)= \sum_{2j+\frac{k}{\beta}<q} \rho^{2j+\frac{k}{\beta}} \left( a_{j,k,\sigma}(\xi) \cos k\theta + b_{j,k,\sigma}(\xi) \sin k\theta \right) + O(\rho^q)\qquad \text{for} \quad  \rho<1/2.
\end{equation}
Moreover, $a_{j,k}(\xi)$, $b_{j,k}(\xi)$, $a_{j,k,\sigma}(\xi)$, $b_{j,k,\sigma}(\xi)$ and the constants in the definition of $O(q)$ are uniformly (independent of $\abs{\xi}<\frac{1}{2}$) bounded by a multiple of $\norm{u}_{C^0(\hat{B}_2)}$.
\end{lem}
For completeness, we include a proof of Lemma \ref{lem:rhogood} in the appendix.

\begin{rem}
	Please note the difference between $a_{j,k,\sigma}$ in this lemma and $a_{j,k}^\sigma$ in Proposition \ref{prop:stronghf}. Lemma \ref{lem:rhogood} does not claim any relation between $a_{j,k}(\xi)$ and $a_{j,k,\sigma}(\xi)$. It will be clear in a minute that
	\begin{equation*}
		a_{j,k,\sigma}(\xi)= \partial_\xi^\sigma a_{j,k}.
	\end{equation*}
	The same applies to $b_{j,k}$ and $b_{j,k,\sigma}$.
\end{rem}

Given Lemma \ref{lem:rhogood} and Lemma \ref{lem:xigood}, we claim that:

{\bf Claim:} the $a_{j,k}(\xi)$ and $b_{j,k}(\xi)$ in \eqref{eqn:expansionu} are smooth functions of $\xi$ on $\set{\abs{\xi}<1/2}$. Moreover, for any multi-index $\sigma$,
\begin{equation}
	\sup_{\abs{\xi}<1/2} \abs{\partial_\xi^\sigma a_{j,k}} + \abs{\partial_\xi^\sigma b_{j,k}} \leq C(j,k,\sigma) \norm{u}_{C^0(\hat{B}_2)}.
	\label{eqn:claim}
\end{equation}

Proposition \ref{prop:stronghf} follows from the claim, because we can expand $a_{j,k}(\xi)$ and $b_{j,k}(\xi)$ (in \eqref{eqn:expansionu}) into the sum of a Taylor polynomial of $\xi$ and a remainder $O(\abs{\xi}^q)$.  Then the $O(d^q)$ in Proposition \ref{prop:stronghf} is the sum of $O(\rho^q)$ in \eqref{eqn:expansionu}, a sum of $X_\beta$-monomials with degree no less than $q$ and
\begin{equation*}
	\rho^{2j+\frac{k}{\beta}} O(\abs{\xi}^q), \qquad 2j+\frac{k}{\beta}<q
\end{equation*}
in the expansion of $a_{j,k}(\xi)$ and $b_{j,k}(\xi)$.

\vskip 0.5cm

For the proof of the claim, we start with $a_{0,0}(\xi)$. By \eqref{eqn:expansionu},
\begin{equation*}
	a_{0,0}(\xi)=\lim_{\rho\to 0} u.
\end{equation*}
By Lemma \ref{lem:xigood} $\abs{\partial_\xi^\sigma u}$ is uniformly bounded for any $\abs{\xi}<\frac{1}{2}$, hence the convergence is in fact in $C^l$ for any $l$, then our claim for $a_{0,0}$ follows.

Now for a fixed $\sigma$, \eqref{eqn:expansionusigma} gives
\begin{equation*}
	a_{0,0,\sigma}(\xi)=\lim_{\rho\to 0} \partial_\xi^\sigma u.
\end{equation*}
As before, since $\lim_{\rho\to 0} u$ is in $C^l$ topology for any $l>0$, we have
\begin{equation*}
	\lim_{\rho\to 0} \partial_\xi^\sigma u = \partial_\xi^\sigma (\lim_{\rho\to 0} u),
\end{equation*}
which implies that
\begin{equation}\label{eqn:tricky}
	a_{0,0,\sigma}(\xi)=\partial_\xi^\sigma a_{0,0}(\xi).
\end{equation}

Since our claim for $a_{0,0}$ is proved, we may assume that it is zero at the very beginning of the proof by replacing $u$ with $u-a_{0,0}(\xi)$. Thanks to \eqref{eqn:tricky}, a consequence of this assumption is that
\begin{equation}\label{eqn:a00}
	a_{0,0,\sigma}\equiv 0,\qquad  \forall \sigma.
\end{equation}

Next, suppose that $2j_1+\frac{k_1}{\beta}$ is the next (smallest) nonzero power in the expansion, namely, $j_1=1$ and $k_1=0$ when $\beta<1/2$ or $j_1=0$ and $k_1=1$ when $\beta\geq 1/2$. \footnote{Think about $\beta=1/2$. Annoying discussion.}\eqref{eqn:a00} and \eqref{eqn:expansionusigma} imply
\begin{equation*}
	\frac{\partial_\xi^\sigma u}{\rho^{2j_1+\frac{k_1}{\beta}}}
\end{equation*}
is uniformly bounded (w.r.t. $\xi$) by a constant depending on $\sigma$. Hence, the convergence
\begin{equation*}
	a_{j_1,k_1}(\xi) = \lim_{\rho\to 0} \frac{1}{\pi}\int_0^{2\pi} \frac{u}{\rho^{2j_1+\frac{k_1}{\beta}}} \cos k_1\theta d\theta
\end{equation*}
is uniform in any $C^l$ topology. This implies that our claim for $a_{j_1,k_1}$ holds, which enables us to assume $a_{j_1,k_1}(\xi)=0$ at the beginning. Notice that we also have (by the same reason)
\begin{equation*}
	\partial_\xi^\sigma a_{j_1,k_1}(\xi) = a_{j_1,k_1,\sigma}(\xi),
\end{equation*}
so that we can repeat the argument to prove the claim for any $a_{j,k}$ and $b_{j,k}$.

\subsection{The approximating $X_\beta$-polynomial is still harmonic}\label{sub:approx}

We prove in this section
\begin{prop}
	\label{prop:truncation} Suppose that $u$ satisfies the assumptions of Proposition \ref{prop:stronghf} and therefore has an expansion given by \eqref{eqn:hfexpansion}. Then
	\begin{equation}\label{eqn:truncation}
		\triangle_\beta \left[ \sum_{2j+\frac{k}{\beta}+\abs{\sigma}<q} \rho^{2j+\frac{k}{\beta}} \left( a_{j,k}^\sigma \cos k\theta  +  b_{j,k}^\sigma  \sin k\theta \right) \xi^\sigma \right] =0.
	\end{equation}
\end{prop}
\begin{proof}
	To prove this proposition, it suffices to justify that the $\triangle_\beta$ of the remainder $O(d^q)$ (in \eqref{eqn:hfexpansion}) is an $O(d^{q-2})$. In fact, letting $T$ be the $X_\beta$-polynomial in \eqref{eqn:truncation} (or equivalently, in \eqref{eqn:hfexpansion}), we have
	\begin{equation}\label{eqn:dueto}
		\triangle_\beta u = \triangle_\beta T + O(d^{q-2})=0,
	\end{equation}
	{\bf if} our claim for $O(d^q)$ holds. One can check that $\triangle_\beta T$ is an $X_\beta$-polynomial of degree smaller than $q-2$, then it vanishes due to \eqref{eqn:dueto}.

	For the claimed property of $O(d^q)$, recall that, by the proof of Proposition \ref{prop:stronghf}, it is the sum of
	\begin{enumerate}[(1)]
		\item $\rho^{2j+\frac{k}{\beta}}\cos k\theta \xi^\sigma$ for $2j+\frac{k}{\beta}+\abs{\sigma}>q$; (there is a similar term with $\sin$ replacing $\cos$)

		\item $\rho^{2j+\frac{k}{\beta}}\cos k\theta O(\abs{\xi}^q)$;

		\item the $O(\rho^q)$ term in \eqref{eqn:expansionu}.
	\end{enumerate}
	It is trivial that the desired property is true for functions in (1). For (2), notice that $O(\abs{\xi}^q)$ comes from the Taylor expansion of $a_{j,k}(\xi)$ and $b_{j,k}(\xi)$ that are smooth in $\xi$ and hence $\partial_\xi^2 O(\abs{\xi}^q)$ is $O(\abs{\xi}^{q-2})$. Given this, it is straightforward to check that
	\begin{equation*}
		\triangle_\beta \left( \rho^{2j+\frac{k}{\beta}}\cos k\theta O(\abs{\xi}^q) \right) = O(d^{q-2}), \qquad \text{where} \quad d^2 =\rho^2 + \abs{\xi}^2.
	\end{equation*}

	It remains to study the remainder $O(\rho^q)$ in \eqref{eqn:expansionu}. In the proof of Proposition \ref{prop:stronghf}, we have shown that $a_{j,k}(\xi)$ and $b_{j,k}(\xi)$ in \eqref{eqn:expansionu} are smooth functions and for any multi-index $\sigma$,
$$\partial_\xi^\sigma a_{j,k}(\xi)=a_{j,k,\sigma}(\xi)\quad \text{and} \quad
\partial_\xi^\sigma b_{j,k}(\xi)=b_{j,k,\sigma}(\xi).
$$
By comparing \eqref{eqn:expansionu} and \eqref{eqn:expansionusigma}, we notice that any $\xi$-derivative of the $O(\rho^q)$ in \eqref{eqn:expansionu} is still a $O(\rho^q)$.

Since $u$ is a harmonic function, it is not difficult to prove by the interior Schauder estimate and Lemma \ref{lem:xigood} that
\begin{equation}\label{eqn:weighted}
	\sup_{\hat{B}_1\setminus \mathcal{S}} \abs{(\rho \partial_\rho)^{k_1} (\partial_\theta)^{k_2} u} \leq C(k_1,k_2).
\end{equation}
Using equation \eqref{eqn:expansionu}, the estimate \eqref{eqn:weighted} holds for $O(\rho^q)$, which implies that
\begin{equation*}
	\left( \partial_\rho^2 + \frac{1}{\rho}\partial_\rho + \frac{1}{\beta^2 \rho^2}\partial_\theta^2 \right) O(\rho^q) = O(\rho^{q-2}).
\end{equation*}
This concludes the proof of Proposition \ref{prop:truncation} by noticing that $\rho\leq d$ and $O(\rho^{q-2})$ is $O(d^{q-2})$.
\end{proof}
\section{Generalized H\"older space on $X_\beta$}\label{sec:Holder}
In this section, we define the $X_\beta$ counterpart of $C^{k,\alpha}$ function on $\Real^n$ and compare it with Donaldson's $C^{2,\alpha}_\beta$ space when $0<\beta<1$ and $0<\alpha<\min \set{\frac{1}{\beta}-1,1}$.

\subsection{Formal discussion}\label{subsec:formal}
The basic idea of our definition as illustrated by Section \ref{sec:space} is to use generalized `polynomial' to describe the regularity near a singular point. In this subsection, we are concerned with the question of what is the correct `polynomial' for $X_\beta$.

Recall that in Proposition \ref{prop:stronghf}, for the expansion of harmonic functions, we defined $X_\beta$-polynomials, which are finite linear combinations of
\begin{equation*}
	\rho^{2j+\frac{k}{\beta}}\cos k\theta \xi^\sigma,  \rho^{2j+\frac{k}{\beta}}\cos k\theta \xi^\sigma
\end{equation*}
where $j,k\in \mathbb N \cup \set{0}$ and $\sigma$ is a multi-index of $\Real^{n-2}$.

$X_\beta$-polynomials are not enough for the study of more complicated PDE solutions, because the product of two $X_\beta$-polynomials is {\bf not} $X_\beta$-polynomial. This motivates the following definition.

\begin{defn}
	Suppose that $j,k,m\in \mathbb N\cup \set{0}$ satisfy $\frac{k-m}{2}\in \mathbb N\cup \set{0}$ and $\sigma$ is a multi-index of dimension $n-2$. The functions
	\begin{equation*}
		\rho^{2j+\frac{k}{\beta}}\cos m\theta \xi^\sigma,\rho^{2j+\frac{k}{\beta}}\sin m\theta \xi^\sigma
	\end{equation*}
	are called (monic) {\bf $\mathcal T$-monomials} of degree $2j+\frac{k}{\beta}+\abs{\sigma}$. A {\bf $\mathcal T$-polynomial} is a finite linear combination of $\mathcal T$-monomials.
\end{defn}
It is elementary to check that the product of $\mathcal T$-polynomials is $\mathcal T$-polynomial.

\begin{rem}
Our previous experience in PDE's with conical singularity suggests that	the regularity of solutions near a singular point (like the cone singularity in $X_\beta$) depends both on the singularity of space and on the type of PDE that we are working with.

While the above definition of $\mathcal T$-polynomial suffices (see \cite{yin2016analysis}) for the study of nonlinear equations like
\begin{equation*}
	\triangle_\beta u =F(u).
\end{equation*}
It is not enough for the conical complex Monge-Ampere equation studied in \cite{yin2016expansion}. To minimize the difficulty of understanding this paper, we refrain from working in that generality. Instead, we will list below a family of properties that should be satisfied by $\mathcal T$-polynomials. It will be clear in the proof that follows, if these properties hold, the main result of this paper remains true for other definitions of $\mathcal T$-polynomials.
\end{rem}

Our definition of $\mathcal T$-polynomial satisfies a family of properties, which we summarize in the form of a lemma.

\begin{lem}
	\label{lem:PPP}
(P1) Let $f$ be a monic $\mathcal T$-monomial of degree $q$. For any $q'<q$, $l\in \mathbb N\cup \set{0}$ and any point $x\in X_\beta$ satisfying $\rho(x)<1/2$ and $\xi(x)=0$,
\begin{equation*}
	\norm{S_x(f)}_{C^l(\tilde{B})}\leq C(q,l,q') \rho(x)^{q'}.
\end{equation*}

(P2) Let $f$ be a monic $\mathcal T$-monomial of degree $q$. There is a constant depending only on $q$ and $\delta>0$ such that for any $l \in \mathbb N\cup \set{0}$,
\begin{equation*}
	\norm{f}_{C^l(\Omega_\delta \cap \hat{B}_1)}\leq C(q,l,\delta).	
\end{equation*}

(P3) Let $f$ be a $\mathcal T$-polynomial of degree $q$. There is a constant $C$ depending only on $q$ such that
\begin{equation*}
	\norm{f}_{C^{l}(\tilde{B}_{1/8}(z))} \leq C(l,q) \norm{f}_{C^0(\tilde{B}_{1/8}(z))}
\end{equation*}
for any $l \in \mathbb N\cup \set{0}$ and all $z\in \tilde{B}_{1/2}$.

(P4) Let $f$ be any $\mathcal T$-polynomial of degree $q$. There exists a $\mathcal T$-polynomial $u$ of degree $q+2$ (not unique) such that
\begin{equation*}
	\triangle_\beta u =f.
\end{equation*}
Moreover, the coefficients of $u$ are bounded by a multiple (depending on $q$) of the coefficients of $f$.
\end{lem}
\begin{proof}
(P1) and (P2) can be checked by direct computation. Notice that (P1) is true even if $q'=q$ and we have stated it in this weaker form, because we have in our mind 	monomials involving $\log$ terms (for example $\rho \log \rho$), which may appear in applications.

For (P3), recall that the number of (monic) $\mathcal T$-monomials with degree no more than $q$ is finite and that they are linearly independent functions so that the $C^0$ norm of a $\mathcal T$-polynomial on any open set bounds every coefficient.

The proof of (P4) is an induction argument based on the fact that $\triangle_\beta \rho^{\gamma+2}\cos m\theta \xi^\sigma$ is a linear combination of $\rho^\gamma \cos m\theta \xi^\sigma$ and a $\mathcal T$-polynomial whose order in $\xi$ is smaller than $\abs{\sigma}$ by $2$.
\end{proof}

\subsection{The definition.}

In this section, we define the generalized H\"older space $\mathcal U^q$. We assume that $q$ is any positive number that is {\bf not in}
\begin{equation*}
	\mathcal D=\set{j+\frac{k}{\beta}|\quad j,k\in \mathbb N\cup \set{0}},
\end{equation*}
which is the set of degrees of $\mathcal T$-monomials.

The overall idea is to require that $u$ is $C^{k,\alpha}$ in the usual sense in $\Omega_\delta \cap \hat{B}_1$ and for each $x\in \mathcal S\cap \hat{B}_1$, $u$ has some uniform expansion (using $\mathcal T$-polynomials) in a ball of size $\delta$.

\begin{defn}
	\label{defn:big}
Suppose that $q=k+\alpha$ for some $k\in \mathbb N\cup \set{0}$ and $\alpha\in (0,1)$. $u$ is said to be in $\mathcal U^q(\hat{B}_1)$ if and only if there is some $\Lambda>0$ such that

(H1) $u$ is $C^{k,\alpha}$ on $\hat{B}_1\cap \Omega_\delta$ and
\begin{equation*}
	\norm{u}_{C^{k,\alpha}(\hat{B}_1\cap \Omega_\delta)}\leq \Lambda.
\end{equation*}

(H2) For each $x\in \mathcal S\cap \hat{B}_1$, there is a $\mathcal T$-polynomial $P_x$ such that (i) the degree of $P_x$ is smaller than $q$ and the coefficients of $P_x$ is bounded by $\Lambda$ and (ii)
\begin{equation*}
u(x+y)=P_x(y) + O(d(y)^q),\quad \forall y\in \hat{B}_\delta,
\end{equation*}
where $O(d(y)^q)$ above satisfies
\begin{equation*}
	\abs{O(d(y)^q)} \leq \Lambda d(y)^q.
\end{equation*}

(H3) For each $x\in S_\delta\cap \hat{B}_1$, let $\tilde{x}$ be the projection of $x$ to $\mathcal S$.
\begin{equation}\label{eqn:h3}
	\norm{S_x(u-P_{\tilde{x}}(\cdot -\tilde{x}))}_{C^{k,\alpha}(\tilde{B}_{1/2})}\leq \Lambda \rho(x)^q.
\end{equation}
\end{defn}
The infimum of $\Lambda$ such that (H1-H3) hold for some $u\in \mathcal U^q(\hat{B}_1)$ is defined to be the norm of $u$, denoted by $\norm{u}_{\mathcal U^q(\hat{B}_1)}$.

(H1-H2) is in line with the overall idea mentioned before Definition \ref{defn:big}. (H3) is necessary to describe the behavior of the functions in $\mathcal U^q(\hat{B}_1)$ at those points that are closer and closer to the singular set. Definition \ref{defn:big} may look unusual, but it will be justified when we show that Donaldson's $C^{2,\alpha}_\beta$ is a special case of $\mathcal U^q$ in Section \ref{subsec:compare} and when we prove a Schauder estimate in Section \ref{sec:schauder}.

Several remarks are helpful in understanding the defnition.

\begin{rem}\label{rem:smaller}
	We need to check that for $0<q_1<q_2$ ($q_1,q_2\notin \mathcal D$), $\mathcal U^{q_2}(\hat{B}_1)\subset \mathcal U^{q_1}(\hat{B}_1)$, which is not totally trivial from the definition above. To see this, let $u$ be in $\mathcal U^{q_2}(\hat{B}_1)$, it suffices to check (H2) and (H3) in the definition of $\mathcal U^{q_1}(\hat{B}_1)$. For (H2), let $x\in \mathcal S\cap \hat{B}_1$, then there is some $\mathcal T$-polynomial $P_x$ such that
	\begin{equation*}
		u(x+y)=P_x(y)+O(d(y)^{q_2}) \qquad \forall y\in \hat{B}_\delta.
	\end{equation*}
	Let $P'_x$ be the part of $P_x$ that involves only monomials of degree smaller than $q_1$. Let $Q_x=P_x-P'_x$. Obviously,
	\begin{equation*}
		u(x+y)=P'_x(y) + O(d(y)^{q_1}) \qquad \forall y\in \hat{B}_\delta.
	\end{equation*}
	For (H3), it suffices to check that for each $x\in \mathcal S_\delta\cap \hat{B}_1$ with $\tilde{x}$ being its projection to $\mathcal S$, $Q_{\tilde{x}}$ satisfies
	\begin{equation}\label{eqn:checkQ}
		\norm{S_x(Q_{\tilde{x}}(\cdot -\tilde{x}))}_{C^{k_1,\alpha_1}(\tilde{B}_{1/2})}\leq C\Lambda \rho(x)^{q_1},
	\end{equation}
	if $q_1=k_1+\alpha_1$.
	In fact, $Q_{\tilde{x}}$ is a finite linear combination of $\mathcal T$-monomials of degree strictly larger than $q_1$ and (H2) implies that the coefficients of the combination are bounded by $\Lambda$. Hence, \eqref{eqn:checkQ} follows from (P1).
\end{rem}

\begin{rem}
	\label{rem:delta}
	Due to (P2) above, the definition of $\mathcal U^q(\hat{B}_1)$ is independent of the constant $\delta$. A different choice of $\delta$ yields an equivalent norm $\norm{\cdot}_{\mathcal U^q}$.
\end{rem}

\begin{rem}
	\label{rem:final}
Finally, we remark that \eqref{eqn:h3} in (H3) of Definition \ref{defn:big} can be replaced by
\begin{equation}\label{eqn:h3p}
	\norm{S_x(u-P_{\tilde{x}}(\cdot-\tilde{x}))(z)}_{C^{k,\alpha}(\tilde{B}_{3/8})}\leq \Lambda \rho(x)^q.
\end{equation}
\end{rem}
This may look plausible, however, we give a detailed proof, because we shall need it explicitly in the proof of Theorem \ref{thm:schauder}. In the sequel, we shall use (H3') for the assumption (H3) with \eqref{eqn:h3} replaced by \eqref{eqn:h3p}. Assume that we have a function $u$ satisfying (H1), (H2) and (H3').

Since (H3) only matters when $\rho(x)$ is small, we fix $x\in S_{\delta/2} \cap \hat{B}_1$. By (H2),
\begin{equation}\label{eqn:SPx}
	\sup_{\tilde{B}} \abs{S_x(u-P_{\tilde{x}}(\cdot-\tilde{x}))} = \sup_{\hat{B}_{c_\beta \rho(x)}(x)} \abs{u- P_{\tilde{x}}(\cdot-\tilde{x})} \leq C\Lambda \rho(x)^q.
\end{equation}
For any $y\in \hat{B}_{c_\beta \rho(x)/2}(x)=\tilde{B}_{1/2}(x)$, (H3') implies that
\begin{equation*}
	\norm{S_y (u-P_{\tilde{y}}(\cdot-\tilde{y}))(z)}_{C^{k,\alpha}(\tilde{B}_{3/8})}\leq \Lambda \rho(y)^q \leq C\Lambda \rho(x)^q.
\end{equation*}
By the definition of $S_x$ and $S_y$,
\begin{equation}\label{eqn:Sxy}
	\begin{split}
		S_x(u-P_{\tilde{y}}(\cdot-\tilde{y}))(z) &= (u-P_{\tilde{y}}(\cdot-\tilde{y}))(\Psi_x^{-1}(z)) \\
	&= (u-P_{\tilde{y}}(\cdot-\tilde{y}))(\Psi_y^{-1} \circ \Psi_y\circ \Psi_x^{-1}(z)) \\
	&= S_y(u-P_{\tilde{y}}(\cdot-\tilde{y})) (\Psi_y\circ \Psi_x^{-1}(z)).
	\end{split}
\end{equation}
If $z=(\rho_z,\theta_z,\xi_z)$, we compute explicitly
\begin{equation}\label{eqn:psi}
	\begin{split}
	\Psi_y\circ \Psi_x^{-1} (z) &= \Psi_y (\rho_x \rho_z, \theta_x+\theta_z, \xi_x+\rho_x\xi_z) \\
	&= ( \frac{\rho_x}{\rho_y}\rho_z, \theta_z +\theta_x-\theta_y, \frac{\rho_x \xi_z+\xi_x-\xi_y}{\rho_y}).
	\end{split}
\end{equation}
Obviously when $z=\Psi_x(y)\in \tilde{B}_{1/2}$, $\Psi_y \circ \Psi_x^{-1}(z)=\tilde{x}_0$. Since $y\in \hat{B}_{c_\beta \rho(x)/2}(x)$ (so that $1/2< \rho_x/\rho_y<2$), we know $\Psi_y\circ \Psi_x^{-1}$ is a Lipschitz map with Lipschitz constant smaller than $2$, which implies that
\begin{equation*}
	\Psi_y\circ \Psi_x^{-1} (\tilde{B}_{1/8}(z))\subset \tilde{B}_{3/8}.
\end{equation*}
Noticing that the map $\Psi_y\circ \Psi_x^{-1}$ is uniformly bounded (for all $y\in \tilde{B}_{1/2}(x)$) in any $C^k$ norm on $\tilde{B}_{1/8}(z)$, \eqref{eqn:Sxy} implies that
\begin{equation}\label{eqn:SPy}
	\norm{S_x(u-P_{\tilde{y}}(\cdot-\tilde{y}))}_{C^{k,\alpha}(\tilde{B}_{1/8}(z))} \leq C\Lambda \rho(x)^q.
\end{equation}
We claim that
\begin{equation}\label{eqn:SP}
	\norm{S_x(u-P_{\tilde{x}}(\cdot-\tilde{x}))}_{C^{k,\alpha}(\tilde{B}_{1/8}(z))} \leq C\Lambda \rho(x)^q.
\end{equation}
(H3) follows from \eqref{eqn:SP} because $y$ is any point in $\tilde{B}_{1/2}(x)$, $z=\Psi_x(y)$ and
\begin{equation*}
	\tilde{B}_{1/2} \subset \bigcup_{y\in \tilde{B}_{1/2(x)}} \tilde{B}_{1/8}(z).
\end{equation*}
By comparing \eqref{eqn:SPy} and \eqref{eqn:SP}, the proof of the above claim reduces to
\begin{equation}\label{eqn:down1}
	\norm{S_x(P_{\tilde{y}}(\cdot-\tilde{y})-P_{\tilde{x}}(\cdot-\tilde{x}))}_{C^{k,\alpha}(\tilde{B}_{1/8}(z))} \leq C\Lambda \rho(x)^q.
\end{equation}
Since $S_x(P_{\tilde{y}}(\cdot-\tilde{y})-P_{\tilde{x}}(\cdot-\tilde{x}))$ is a $\mathcal T$-polynomial with degree smaller than $q$, (P3) implies that \eqref{eqn:down1} follows from 
\begin{equation}\label{eqn:down2}
	\norm{S_x(P_{\tilde{y}}(\cdot-\tilde{y})-P_{\tilde{x}}(\cdot-\tilde{x}))}_{C^{0}(\tilde{B}_{1/8}(z))} \leq C\Lambda \rho(x)^q.
\end{equation}
Recall that $z=\Psi_x(y)$, hence $\Psi_x$ maps $\tilde{B}_{1/8}(z)$ to $\tilde{B}_{1/8}(y)\subset \tilde{B}(x)$. (H2) implies respectively
\begin{equation*}
	\norm{u-P_{\tilde{y}(\cdot-\tilde{y})}}_{C^0(\tilde{B}_{1/8}(y))} \leq C\Lambda \rho(y)^q \leq C\Lambda \rho(x)^q
\end{equation*}
and
\begin{equation*}
	\norm{u-P_{\tilde{x}(\cdot-\tilde{x})}}_{C^0(\tilde{B}(x))}\leq C\Lambda \rho(x)^q.
\end{equation*}
Then \eqref{eqn:down2} follows from a combination of the above two inequalities and the fact that $\tilde{B}_{1/8}(y)\subset \tilde{B}(x)$.

\subsection{Comparison with Donaldson's $C^{2,\alpha}_\beta$ space}\label{subsec:compare}
Donaldson's definition requires both $0<\beta<1$ and $\alpha< \min\set{\frac{1}{\beta}-1,1}$, which we assume in this subsection only. It is the purpose of this subsection to show that the space $\mathcal U^{2+\alpha}$ is equivalent to $C^{2,\alpha}_\beta$ when the above restriction to $\beta$ and $\alpha$ applies.

In this case, the only monic $\mathcal T$-monomials whose order is smaller than $2+\alpha$ are
\begin{equation}\label{eqn:monic}
	1, \rho^{\frac{1}{\beta}}\cos \theta, \rho^{\frac{1}{\beta}}\sin \theta,\rho^2, \xi, \xi^2.
\end{equation}

First, let's recall the definition of $C^{2,\alpha}_\beta$ space defined by Donaldson. It is defined to be the set of functions satisfying
\begin{enumerate}[(D1)]
	\item $u$ is in $C^\alpha(\hat{B}_1)$;
	\item $\partial_\rho u$, $\frac{1}{\rho} \partial_\theta u$ and $\partial_\xi u$ are in $C^\alpha(\hat{B}_1)$;
	\item $\partial_\xi^2 u$, $\partial_\xi \partial_\rho u$, $\frac{1}{\rho}\partial_\xi \partial_\theta u$ and $\tilde{\triangle}_\beta u$ are in $C^{\alpha}(\hat{B}_1)$, where $\tilde{\triangle}_\beta= \partial_\rho^2 + \frac{1}{\rho}\partial_\rho + \frac{1}{\beta^2 \rho^2} \partial_\theta^2$ is the Laplacian on the cone surface of cone angle $2\pi \beta$.
\end{enumerate}
Here $C^\alpha(\hat{B}_1)$ is the space of H\"older continuous functions with respect to the distance $d$ of $X^\beta$. Moreover, the $C^{2,\alpha}_\beta$ norm is the sum of all $C^\alpha$ norms mentioned above.

We start the comparison with the following lemma,
\begin{lem}
	\label{lem:c1} 
	(i) If $u$ is $C^\alpha(\hat{B}_2)$, then $u$ is in $\mathcal U^{\alpha}(\hat{B}_1)$ with
	\begin{equation}\label{eqn:c1}
		\norm{u}_{\mathcal U^\alpha(\hat{B}_1)}\leq C \norm{u}_{C^\alpha(\hat{B}_2)};
	\end{equation}
	and (ii) if $u$ is in $\mathcal U^\alpha (\hat{B}_2)$, then $u$ is in $C^\alpha(\hat{B}_1)$ with
	\begin{equation}\label{eqn:c2}
		\norm{u}_{C^\alpha(\hat{B}_1)}\leq C \norm{u}_{\mathcal U^\alpha(\hat{B}_2)}.
	\end{equation}
\end{lem}

\begin{proof}
	(i) We notice that (H1) is trivial and (H2) is just the definition of H\"older continuous if we choose $P_x$ to be the constant $u(x)$. For $x\in \mathcal S_\delta\cap \hat{B}_1$ and any $z\in \tilde{B}_{1/2}$, since $P_{\tilde{x}}$ is the constant $u(\tilde{x})$, we have
\begin{eqnarray*}
	S_x(u-P_{\tilde{x}}(\cdot - \tilde{x}))(z)&=&  (u-P_{\tilde{x}}(\cdot-\tilde{x}))(\Psi_x^{-1}(z))\\
	&=& u(\Psi_x^{-1}(z)) - u(\tilde{x}).
\end{eqnarray*}
For any $z\in \tilde{B}_{1/2}$, the distance from $\Psi_x^{-1}(z)$ to $\tilde{x}$ is at most $2\rho(x)$ and hence
\begin{equation}\label{eqn:c11}
	\norm{S_x(u-P_{\tilde{x}}(\cdot-\tilde{x}))}_{C^0(\tilde{B}_{1/2})}\leq C \norm{u}_{C^\alpha(\hat{B}_2)}\rho(x)^\alpha.
\end{equation}
For any $z_1,z_2$ in $\tilde{B}_{1/2}$,
\begin{equation}
	\label{eqn:c12}
	\begin{split}
		&\abs{S_x(u-P_{\tilde{x}}(\cdot-\tilde{x}))(z_1)-S_x(u-P_{\tilde{x}}(\cdot-\tilde{x}))(z_2)}\\
	=&  \abs{u(\Psi_x^{-1}(z_1)) - u(\Psi_x^{-1}(z_2))} \\
	\leq& \norm{u}_{C^\alpha(\hat{B}_2)} \left( \rho(x)d(z_1,z_2)\right)^\alpha.
	\end{split}
\end{equation}
Hence, (H3) and then \eqref{eqn:c1} follow from \eqref{eqn:c11} and \eqref{eqn:c12}.

(ii) By (H1), it suffices to consider $x_1$ and $x_2$ in $\mathcal S_\delta\cap \hat{B}_1$. Assume that $\rho(x_1)\geq \rho(x_2)$.

{\bf If $d(x_1,x_2)\geq \frac{c_\beta \rho(x_1)}{2}$:}

Let $\tilde{x}_1$ and $\tilde{x}_2$ be the projections of $x_1$ and $x_2$ to $\mathcal S$ respectively. The triangle inequality and (H2) imply that
\begin{eqnarray*}
	\abs{u(x_1)-u(x_2)} &\leq& \abs{u(x_1)-u(\tilde{x}_1)} + \abs{u(x_2)-u(\tilde{x}_2)} + \abs{u(\tilde{x}_1)- u(\tilde{x}_2)} \\
	&\leq& \norm{u}_{\mathcal U^\alpha (\hat{B}_2)} \left( \rho(x_1)^\alpha + \rho(x_2)^\alpha + d(\tilde{x}_1,\tilde{x}_2)^\alpha \right) \\
	&\leq & C \norm{u}_{\mathcal U^\alpha (\hat{B}_2)} d(x_1,x_2)^\alpha.
\end{eqnarray*}

{\bf If $d(x_1,x_2)< \frac{c_\beta \rho(x_1)}{2}$:} denote $\Psi_{x_1}(x_2)$ by $z$, which is a point in $\tilde{B}_{1/2}$.

\begin{eqnarray*}
	\abs{u(x_1)-u(x_2)} &=& \abs{S_{x_1}(u)(\tilde{x}_0) - S_{x_1}(u)(z)} \\
	&\leq& \abs{S_{x_1}(u-P_{\tilde{x}_1}(\cdot -\tilde{x}_1))(\tilde{x}_0)-S_{x_1}(u-P_{\tilde{x}_1}(\cdot -\tilde{x}_1))(z) } \\
	&\leq & C \norm{u}_{\mathcal U^\alpha(\hat{B}_2)} \rho(x_1)^\alpha d(\tilde{x}_0,z)^\alpha \\
	&\leq& C \norm{u}_{\mathcal U^\alpha(\hat{B}_2)} d(x_1,x_2)^\alpha.
\end{eqnarray*}
Here in the second line above, we used (H3).
\end{proof}

To conclude this section, we prove
\begin{lem}
	\label{lem:c2} If $u$ is in $\mathcal U^{2+\alpha}(\hat{B}_2)$, then $u$ is in $C^{2,\alpha}_\beta(\hat{B}_1)$ and
	\begin{equation*}
		\norm{u}_{C^{2,\alpha}_\beta(\hat{B}_1)} \leq C \norm{u}_{\mathcal U^{2+\alpha}(\hat{B}_2)}.
	\end{equation*}
\end{lem}
\begin{rem}
	It is natural to ask about the other direction of Lemma \ref{lem:c2}. While it is possible to give a direct proof, we omit it because it follows from Lemma \ref{lem:c1} and the Schauder estimate (Theorem \ref{thm:schauder}), which is to be proved in the next section. So we conclude that $C^{2,\alpha}_\beta$ and $\mathcal U^{2+\alpha}$ are the same (in the sense above).
\end{rem}

\begin{proof}
	The proof consists of several steps.

	{\bf Step 1:}
	By Remark \ref{rem:smaller} and Lemma \ref{lem:c1}, we have $u\in C^\alpha(\hat{B}_1)$.	

	{\bf Step 2:} All derivatives listed in (D2) and (D3) are bounded. Since the proofs are the same, we prove the claim for $\partial_\rho u$ only. Thanks to (H1), it suffices to consider $x\in \mathcal S_\delta\cap \hat{B}_1$. Let $\tilde{x}$ be the projection of $x$ onto $\mathcal S$, then for the $\mathcal T$-polynomial $P_{\tilde{x}}$ (with order smaller than $2+\alpha$) in (H2), we have
	\begin{eqnarray}\label{eqn:decompose}
		\partial_\rho u(x) &=& \partial_\rho \left( u - P_{\tilde{x}}(\cdot-\tilde{x}) \right)(x) + \partial_\rho P_{\tilde{x}} (x-\tilde{x}).
	\end{eqnarray}
	The second term above is bounded by a multiple of $\norm{u}_{\mathcal U^{2+\alpha}(\hat{B}_2)}$, because of (H2) and the fact that the $\partial_\rho$ of each monic $\mathcal T$-monomial (listed in \eqref{eqn:monic}) is bounded. The first term is bounded by
	\begin{equation}\label{eqn:firstterm}
		\frac{1}{\rho} \abs{\nabla S_x(u-P_{\tilde{x}}(\cdot-\tilde{x}))}(\tilde{x}_0),
	\end{equation}
	which is in turn bounded by $C\norm{u}_{\mathcal U^{2+\alpha}(\hat{B}_2)}$ due to (H3).

	{\bf Step 3:}
	All derivatives of $u$ listed in (D2) and (D3) are in $C^\alpha$.

	As before, due to (H1), we may assume that $x_1,x_2\in \mathcal S_\delta\cap \hat{B}_1$ and $\rho(x_1)\geq \rho(x_2)$.

	{\bf If $d(x_1,x_2)< c_\beta\rho(x_1)/2$:}  Let $\tilde{x}_1$ be the projection of $x_1$ onto $\mathcal S$ and $z=\Psi_{x_1}(x_2)\in \tilde{B}_{1/2}$.
	\begin{eqnarray*}
		&& \abs{\partial_\rho u(x_1)- \partial_\rho u(x_2)} \\
		&\leq & \abs{\partial_\rho (u-P_{\tilde{x}_1}(\cdot-\tilde{x}_1)) (x_1)- \partial_\rho (u-P_{\tilde{x}_1}(\cdot-\tilde{x}_1)) (x_2)} + \abs{\partial_\rho P_{\tilde{x}_1}(x_1-\tilde{x}_1) - \partial_\rho P_{\tilde{x}_1}(x_2-\tilde{x}_1) } \\
		&\leq & \frac{1}{\rho(x_1)} \abs{\nabla S_{x_1}(u-P_{\tilde{x}_1}(\cdot-\tilde{x}_1))(\tilde{x}_0)-\nabla S_{x_1}(u-P_{\tilde{x}_1}(\cdot-\tilde{x}_1))(z) } + C\norm{u}_{\mathcal U^{2+\alpha}(\hat{B}_2)} d(x_1,x_2)^\alpha\\
		&\leq & \frac{1}{\rho(x_1)}\cdot C\norm{u}_{\mathcal U^{2+\alpha}(\hat{B}_2)} \rho(x_1)^{2+\alpha} d(\tilde{x}_0,z)^\alpha + C\norm{u}_{\mathcal U^{2+\alpha}(\hat{B}_2)} d(x_1,x_2)^\alpha\\
		&\leq &  C \norm{u}_{\mathcal U^{2+\alpha}(\hat{B}_2)}d(x_1,x_2)^\alpha.
	\end{eqnarray*}
	For the estimate of the second term in the second line above, we check that for each monic monomial $f$ in \eqref{eqn:monic}, there holds
	\begin{equation*}
		\abs{\partial_\rho f(x_1)- \partial_\rho f(x_2)}\leq C d(x_1,x_2)^\alpha.
	\end{equation*}
	For the first term in the third line above, we used (H3).

	Again, the same argument works for all other derivatives in (D2) and (D3) in the case $d(x_1,x_2)<c_\beta\frac{\rho(x_1)}{2}$.

	{\bf If $d(x_1,x_2)\geq c_\beta\rho(x_1)/2$:}
	We study $\partial_\rho u$, $\frac{1}{\rho} \partial_\theta u$ and $\tilde{\triangle}_\beta u$ first. In this case, (H3) implies that
	\begin{equation}\label{eqn:ux1}
		\abs{\partial_\rho (u(x_1)-P_{\tilde{x}_1}(x_1-\tilde{x}_1))}\leq \frac{1}{\rho(x_1)}\abs{\nabla S_{x_1}(u-P_{\tilde{x}_1}(\cdot-\tilde{x}_1))}(\tilde{x}_0) \leq C\norm{u}_{\mathcal U^{2+\alpha}(\hat{B}_2)} \rho(x_1)^{1+\alpha}
	\end{equation}
	and
	\begin{equation}\label{eqn:ux2}
		\abs{\partial_\rho (u(x_2)-P_{\tilde{x}_2}(x_2-\tilde{x}_2))}\leq \frac{1}{\rho(x_2)}\abs{\nabla S_{x_2}(u-P_{\tilde{x}_2}(\cdot-\tilde{x}_2))}(\tilde{x}_0) \leq C\norm{u}_{\mathcal U^{2+\alpha}(\hat{B}_2)} \rho(x_2)^{1+\alpha}.
	\end{equation}
	By checking the monic monomials in \eqref{eqn:monic} one by one, we find that for either $x=x_1$ or $x=x_2$,
	\begin{equation}\label{eqn:ux12}
		\abs{\partial_\rho P_{\tilde{x}}(x-\tilde{x})}\leq C \norm{u}_{\mathcal U^{2+\alpha}(\hat{B}_2)} \rho(x)^\alpha.
	\end{equation}
	By \eqref{eqn:ux1}, \eqref{eqn:ux2} and \eqref{eqn:ux12} and the fact that $\rho(x_2)\leq \rho(x_1)\leq \frac{2}{c_\beta}d(x_1,x_2)$, we get
\begin{eqnarray*}
	\abs{\partial_\rho u(x_1) - \partial_\rho u(x_2)} \leq C\norm{u}_{\mathcal U^{2+\alpha}(\hat{B}_2)} d(x_1,x_2)^\alpha.
\end{eqnarray*}
The same argument works for $\frac{1}{\rho}\partial_\theta u$ and $\tilde{\triangle}_\beta u$.

Now, let's turn to the proof for $\partial_\xi u$, $\partial_\rho\partial_\xi u$, $\frac{1}{\rho}\partial_\theta \partial_\xi u$ and $\partial_\xi^2 u$. We prove for $\partial_\xi^2 u$ only and the same proof works for the other three functions. Similar to \eqref{eqn:decompose}, we have
\begin{equation}
	\partial_\xi^2 u (x) = \partial_\xi^2 (u-P_{\tilde{x}}(\cdot-\tilde{x}))(x) + \partial_\xi^2 P_{\tilde{x}}(x-\tilde{x}).
	\label{eqn:decompose2}
\end{equation}
By \eqref{eqn:decompose2}, we have
\begin{eqnarray*}
	&& \abs{\partial_\xi^2 u(x_1) - \partial_\xi^2 u(x_2)} \\
	&\leq& \abs{\partial_\xi^2 (u-P_{\tilde{x}_1}(\cdot-\tilde{x}_1))(x_1)} + \abs{\partial_\xi^2(u-P_{\tilde{x}_2}(\cdot-\tilde{x}_2))(x_2)} + \abs{\partial_\xi^2 P_{\tilde{x}_1}(x_1-\tilde{x}_1)- \partial_\xi^2 P_{\tilde{x}_2}(x_2-\tilde{x}_2)}.
\end{eqnarray*}
Using (H3) and the inequality $\rho(x_2)\leq \rho(x_1)\leq \frac{2}{c_\beta} d(x_1,x_2)$ as before, we can bound the first two terms by $C \norm{u}_{\mathcal U^{2+\alpha}(\hat{B}_2)} d(x_1,x_2)^\alpha$. It remains to estimate the third term above.
\begin{rem}
	For each $f$ in \eqref{eqn:monic}, the mixed second derivatives $\partial_\rho\partial_\xi$ and $\frac{1}{\rho}\partial_\theta \partial_\xi$ vanish. The proofs are done in these two cases.	
\end{rem}
Again by checking the monic monomials in \eqref{eqn:monic}, we notice that
\begin{equation*}
	\partial_\xi^2 P_{\tilde{x}_i}(x_i-\tilde{x}_i)= \partial_\xi^2 P_{\tilde{x}_i}(0), \qquad i=1,2.
\end{equation*}
The proof will be done, if we can show
\begin{equation}\label{eqn:final}
	\abs{\partial_\xi^2 P_{\tilde{x}_1}(0)- \partial_\xi^2 P_{\tilde{x}_2}(0)}\leq C \abs{\tilde{x}_1-\tilde{x}_2}^\alpha.
\end{equation}
To see this, we define a function $\tilde{u}$ on $\mathcal S\cap \hat{B}_1$ by
\begin{equation*}
	\tilde{u}(\tilde{x})= u(\tilde{x}) = P_{\tilde{x}}(0).
\end{equation*}
By (H2) (restricted to $\mathcal S$ direction), $\tilde{u}$ satisfies the assumptions in Definition \ref{defn:ube} in Section \ref{sec:space}. Proposition \ref{prop:rn} then implies that $\tilde{u}$ is $C^{2,\alpha}$ (on $\Real^{n-2}$) in the usual sense and that
\begin{equation*}
	(\partial_\xi^2 P_{\tilde{x}})(0) = (\partial_\xi^2 \tilde{u})(\tilde{x}),
\end{equation*}
from which \eqref{eqn:final} follows.
\end{proof}
\section{Schauder estimates for $X_\beta$}\label{sec:schauder}
We prove here an estimate that by our perspective should be called the Schauder estimate on $X_\beta$.  Recall that $\mathcal D$ is the set of degrees of all $\mathcal T$-monomials and we have defined generalized H\"older spaces $\mathcal U^q$ only for $q\notin \mathcal D$.
\begin{thm}
	\label{thm:schauder} Assume that $q>0$ and $q,q+2\notin \mathcal D$. Suppose $f\in \mathcal U^{q}(\hat{B}_2)$ and $u$ is a bounded weak solution to
	\begin{equation*}
		\triangle_\beta u =f.
	\end{equation*}
	Then $u$ is in ${\mathcal U}^{q+2}(\hat{B}_1)$ and
	\begin{equation*}
		\norm{u}_{{\mathcal U}^{q+2}(\hat{B}_1)}\leq C\left( \norm{u}_{C^0(\hat{B}_2)} + \norm{f}_{{\mathcal U}^{q}(\hat{B}_2)} \right).
	\end{equation*}
\end{thm}

This is the main result of this paper. The key step in its proof is an analog of Lemma \ref{lem:key}.

\begin{lem}
	\label{lem:hkey}
	For $q>0$ and $q+2 \notin D$, if $f$ is $O(d^q)$ on $\hat{B}_r$, then there exists some $u$ that is $O(d^{q+2})$ and satisfies
	\begin{equation*}
		\triangle_\beta u =f
	\end{equation*}
	and
	\begin{equation*}
		[u]_{{O}_{q+2},\hat{B}_r}\leq C [f]_{{O}_q,\hat{B}_r}.
	\end{equation*}
\end{lem}
The rest of this section consists of two parts. In the first part, we prove Lemma \ref{lem:hkey}, following the proof of Lemma \ref{lem:key} and utilizing many facts about harmonic functions on $X_\beta$ proved in Section \ref{sec:harmonic}. In the second part, we complete the proof of Theorem \ref{thm:schauder}.

\subsection{Proof of Lemma \ref{lem:hkey}}
Without loss of generality, we assume $r=1$.

Setting
\begin{equation*}
	\hat{A}_l:=\hat{B}_{2^{-l}}\setminus \hat{B}_{2^{-l-1}}, \quad \text{for} \quad l=0,1,2,\cdots
\end{equation*}
and
\begin{equation*}
	f_l= f\cdot \chi_{\hat{A}_l},
\end{equation*}
we have
\begin{equation*}
	\abs{f_l}(x)\leq \Lambda_f d(x)^q \chi_{\hat{A}_l}\leq \Lambda_f 2^{-lq}\qquad \text{on} \quad X_\beta,
\end{equation*}
where as before $\Lambda_f=[f]_{\hat{O}_q,\hat{B}_1}$.

Let $w_l$ be the solution to the Poisson equation $\triangle_\beta w_l=f_l$ on $X_\beta$ in Lemma \ref{lem:poisson}, which satisfies that
\begin{equation}
	\sup_{X_\beta} \abs{w_l} \leq C\Lambda_f 2^{-l(q+2)}.
	\label{eqn:hwl}
\end{equation}

Again, $w_l$ is harmonic in $\hat{B}_{2^{-l-1}}$. Let $P_l$ be the $X_\beta$-polynomial in Proposition \ref{prop:stronghf} applied to $w_l$ in $\hat{B}_{2^{-l-1}}$ with $q+2$ replacing $q$.
Our plan is to set
\begin{equation*}
	u_l(x)=w_l(x)-P_l(x)
\end{equation*}
and to show that the series
\begin{equation*}
	u(x)=\sum_{l=0}^\infty u_l(x)
\end{equation*}
converges and gives the solution needed in Lemma \ref{lem:hkey}.
Notice that by Proposition \ref{prop:truncation}, $P_l$ is harmonic on the entire $X_\beta$ and $u_l$ is harmonic outside $\hat{A}_l$.

For that purpose, we need an analog of Lemma \ref{lem:ul}.
\begin{lem}
	\label{lem:hul} 
	There exists a constant $C_q$ depending on $n,\beta$ and $q$ such that

	(i) on $\hat{B}_{2^{-l-1}}$,
	\begin{equation*}
		\abs{u_l(x)} \leq C_q\Lambda_f 2^{( (q+2)^*-(q+2))l}d(x)^{ (q+2)^*};
	\end{equation*}
	Here $(q+2)^*$ is the smallest number in $\mathcal D$ that is larger than $q+2$.

	(ii) on $\hat{A}_l$,
	\begin{equation*}
		\abs{u_l(x)} \leq C_q \Lambda_f 2^{-(q+2)l}
	\end{equation*}
	or equivalently
	\begin{equation*}
		\abs{u_l(x)} \leq C_q \Lambda_f d(x)^{q+2}; 
	\end{equation*}

	(iii) on $\hat{B}_1\setminus \hat{B}_{2^{-l}}$,
	\begin{equation*}
		\abs{u_l(x)} \leq C_q \Lambda_f \left[ 2^{-l(q+2)} + \sum_{2j+\frac{k}{\beta}+ \abs{\sigma}<q+2} 2^{-l(q+2-(2j+\frac{k}{\beta}+\abs{\sigma}) )} d(x)^{2j+\frac{k}{\beta}+\abs{\sigma}} \right].
	\end{equation*}
\end{lem}
\begin{proof}
	The following are immediate corollaries of Proposition \ref{prop:stronghf}.
\begin{enumerate}[(a)]
	\item Recall that $P_l(x)$ is the $X_\beta$-polynomial in Proposition \ref{prop:stronghf} applied to $w_l$, namely,
		\begin{equation*}
			P_l(x) = \sum_{2j+\frac{k}{\beta}+\abs{\sigma}<q+2} \rho^{2j+\frac{k}{\beta}} \left( a_{j,k,\sigma}^l \cos k\theta + b_{j,k,\sigma}^l \sin k\theta \right) \xi^\sigma,
		\end{equation*}
		then
		\begin{equation*}
			\abs{a_{j,k,\sigma}^l} + \abs{b_{j,k,\sigma}^l} \leq C\Lambda_f 2^{-l(q+2-(2j+\frac{k}{\beta}+\abs{\sigma}))};
		\end{equation*}
	\item $u_l(x)$ is ${O}(d^{(q+2)^*})$ in $\hat{B}_{2^{-l-1}}$ and
		\begin{equation*}
			[u_l(x)]_{{O}_{(q+2)^{*}},\hat{B}_{2^{-l-1}}}\leq C\Lambda_f 2^{-l(q+2-(q+2)^{*})}.
		\end{equation*}
\end{enumerate}

Given the $C^0$ bound in \eqref{eqn:hwl}, (a) follows from a scaled version of Proposition \ref{prop:stronghf} applied to the harmonic function $w_l$ on $\hat{B}_{2^{-l-1}}$. For (b), we notice that $u_l$ is exactly the $O(d^q)$ term in \eqref{eqn:hfexpansion} if $q$ in Proposition \ref{prop:stronghf} is taken to be $(q+2)^*$ here. The factor $2^{-l(q+2)}$ comes from the $C^0$ bound of $w_l$ \eqref{eqn:hwl} and $2^{(l(q+2)^*)}$ is due to scaling.

With (a) and (b) above, we notice that (1) is the same as (b); (2) is the same as \eqref{eqn:hwl}; (3) is an easy combination of (a) and \eqref{eqn:hwl}.
\end{proof}

With this lemma, a similar computation as in the $\Real^n$ case verifies that the series $\sum_{l=1}^\infty u_l$ converges and therefore gives the solution we need in Lemma \ref{lem:hkey}.

\subsection{Proof of Theorem \ref{thm:schauder}}

By the usual Schauder estimate, to prove an estimate of $u$ in $\mathcal U^{q+2}(\hat{B}_1)$, we do not need to worry about (H1). For (H2), let $x\in \mathcal S\cap \hat{B}_1$, there is a $\mathcal T$-polynomial $P_f$ (order less than $q$) such that
\begin{equation*}
	f(y)=P_f(y-x) + O(d(x,y)^q),\quad \forall y\in \hat{B}_\delta(x).
\end{equation*}
By our choice of $\mathcal T$ (see (P4) in Lemma \ref{lem:PPP}), there is some $\mathcal T$-polynomial $\tilde{P}_u$ (order less than $q+2$) such that
\begin{equation}\label{eqn:puf}
	\triangle_\beta \tilde{P}_u =P_f.
\end{equation}
Notice that $\tilde{P}_u$ is not uniquely determined by $P_f$, since we may add any harmonic $\mathcal T$-polynomial to it.

If we denote the $O(d(x,y)^q)$ term by $e_f(y)$, Lemma \ref{lem:hkey} implies the existence of $e_u(y)$ defined on $\hat{B}_\delta(x)$ such that
\begin{equation*}
	\triangle_\beta e_u(y) = e_f(y)
\end{equation*}
and
\begin{equation*}
	[e_u]_{O_{q+2},\hat{B}_\delta(x)} \leq C [e_f]_{O_q,\hat{B}_{\delta}(x)}.
\end{equation*}
Therefore, $u-\tilde{P}_u(\cdot-x)-e_u$ is a harmonic function $v$ bounded by $C(\norm{u}_{C^0(\hat{B}_2)}+ \norm{f}_{\mathcal U^q(\hat{B}_2)})$. Proposition \ref{prop:stronghf} implies the existence of some $X_\beta$-polynomial $h_u$ of order less than $q+2$ such that
\begin{equation*}
	v(y)=h_u(y-x)+O(d(x,y)^{q+2}).
\end{equation*}
Then (H2) is verified by setting $P_{\tilde{x}}=\tilde{P}_u + h_u$.

For (H3), let $x\in \mathcal S_\delta$ and $\tilde{x}$ be its projection to $\mathcal S$. On one hand, (H2), which is proved above, implies the existence of $P_{\tilde{x}}$ such that
\begin{equation}\label{eqn:qq1}
	\norm{S_x(u-P_{\tilde{x}}(\cdot-\tilde{x}))}_{C^0(\tilde{B})}\leq C \rho(x)^{q+2}.
\end{equation}
On the other hand, by \eqref{eqn:puf} and the definition of $P_{\tilde{x}}$,
\begin{equation*}
	\triangle_\beta (S_x(u-P_{\tilde{x}}(\cdot-\tilde{x}))) = \rho(x)^{2} S_x (\triangle_\beta (u-P_{\tilde{x}}(\cdot-\tilde{x}))) = \rho(x)^{2} S_x (f-P_f(\cdot-\tilde{x})).
\end{equation*}
By (H3) for $f$ and the above equation,
\begin{equation}\label{eqn:qq2}
	\norm{ \triangle_\beta S_x(u-P_{\tilde{x}}(\cdot-\tilde{x}))}_{C^{k,\alpha}(\tilde{B}_{1/2})}\leq C \rho(x)^{q+2}.
\end{equation}
The usual interior Schauder estimate on $\tilde{B}_{1/2}$, together with \eqref{eqn:qq1} and \eqref{eqn:qq2}, implies that
\begin{equation*}
	\norm{S_x(u-P_{\tilde{x}}(\cdot-\tilde{x}))}_{C^{k+2,\alpha}(\tilde{B}_{3/8})}\leq C \rho(x)^{q+2}.
\end{equation*}
Now the proof of Theorem \ref{thm:schauder} is concluded by Remark \ref{rem:final}.

\appendix

\section{Proof of Lemma \ref{lem:rhogood}}

The proof of this lemma consists of a bootstrapping argument of a family of Poisson equations on cone surface, which was used in \cite{yin2016analysis}. More precisely, by Lemma \ref{lem:xigood}, $\partial_\sigma u$ is harmonic, which implies that
\begin{equation*}
	\tag{$E_\sigma$} \tilde{\triangle}_\beta (\partial_\xi^\sigma u)=- \triangle_\xi (\partial_\xi^\sigma u).
\end{equation*}
Here $\tilde{\triangle}_\beta$ is the Laplacian of the two dimensional cone surface $X_\beta^2$, parametrized by $(\rho,\theta)$ and equipped with the cone metric
\begin{equation*}
	g_\beta^2 = d\rho^2 + \beta^2 \rho^2 d\theta^2.
\end{equation*}
Also by Lemma \ref{lem:xigood}, the right hand side of ($E_\sigma$) is bounded by some constant depending on $\sigma$. It then follows that $\partial_\xi^\sigma u$ (with $\xi$ fixed) is H\"older continuous function with respect to the distance of $X_\beta^2$.  To see this, recall that in terms of the $(u,v)$ coordinates, where $u=\rho\cos \theta$ and $v=\rho\sin \theta$, $\tilde{\triangle}_\beta$ is a uniformly elliptic operator with bounded coefficients and hence the H\"older continuity follows from De Giorgi's iteration. (see \cite{yin2016analysis} for detail). This is the starting point of the bootstrapping.

A function $w$ is said to have $\mathcal T_h$-expansion up to order $q>0$ if
\begin{equation*}
	w(\rho,\theta)= \sum_{2j+\frac{k}{\beta}<q} \rho^{2j+\frac{k}{\beta}}(A_{j,k} \cos k\theta + B_{j,k} \sin k\theta) + O(\rho^q)
\end{equation*}
for $\rho<1/2$. The expansion is said to be bounded by $\Lambda$ if the coefficients $A_{j,k}$, $B_{j,k}$ and the constant in the definition of $O(\rho^q)$ are bounded by $\Lambda$.

\begin{rem}\label{rem:diff}
	Note that the expansion above is different from the one used in \cite{yin2016analysis}, where we also included
	\begin{equation*}
		\rho^{2j+\frac{k}{\beta}}\cos m\theta \quad \text{and} \quad \rho^{2j+\frac{k}{\beta}}\sin m\theta, \qquad \text{for} \quad \frac{k-m}{2}\in \mathbb N\cup \set{0}.
	\end{equation*}
	This is because in \cite{yin2016analysis}, we dealt with nonlinear equations, while here we are essentially working with linear equations. The product of harmonic functions is not necessarily harmonic and hence there is no need to require the formal series to be multiplicatively closed.
\end{rem}

The H\"older continuity of $\partial_\xi^\sigma u$ means that $\partial_\xi^\sigma u(\xi)$ has an expansion up to some order $q\in (0,1)$ uniformly (independent of $\xi$) bounded by $\Lambda=\Lambda(q,\sigma)$. The proof of Lemma \ref{lem:rhogood} is now reduced to the following claim. Notice that we prove \eqref{eqn:expansionu} and \eqref{eqn:expansionusigma} simultaneously.

{\bf Claim.} Let $u$ be a bounded solution to
\begin{equation*}
	\tilde{\triangle}_\beta u =f
\end{equation*}
on the unit ball centered at the unique singular point of $X^2_\beta$. If $f$ has a $\mathcal T_h$-expansion up to order $q$ bounded by $\Lambda$ for $q\ne 2j+\frac{k}{\beta}$ for any $k,j\in \mathbb N\cup \set{0}$, then $u$ has a $\mathcal T_h$-expansion up to order $q+2$ bounded by a multiple of $\Lambda$ and $C^0$ norm of $u$ on the ball.

This is nothing but Lemma 6.9 in \cite{yin2016analysis}. The difference pointed out in Remark \ref{rem:diff} does not cause a problem because for the proof, we only require that for each $\mathcal T_h$-polynomial of order $q'$, there exists a $\mathcal T_h$-polynomial of order $q'+2$ that is mapped to the given one by $\tilde{\triangle}_\beta$.

As a final remark, we notice that Lemma 6.9 in \cite{yin2016analysis} relies on Lemma 6.10 there, which is the precursor of Lemma \ref{lem:key} and Lemma \ref{lem:hkey} in this paper. We find the proof of Lemma 6.10 in \cite{yin2016analysis}, which depends on Fourier series, hard to generalize to higher dimensions. The new proof here of course can be used to prove Lemma 6.10.

\bibliography{foo}
\bibliographystyle{plain}

\end{document}